\documentclass[a4paper,12pt]{amsart}
\usepackage{amssymb}
\usepackage{ifthen}
\usepackage{cite}
 \usepackage[dvips]{graphicx}
\nonstopmode \numberwithin{equation}{section}
\usepackage{amssymb}
\usepackage{ifthen}
\usepackage{graphicx}
\usepackage{amsmath}
\usepackage[T1]{fontenc} 
\usepackage[utf8]{inputenc}
\usepackage[usenames,dvipsnames]{color}
\usepackage{color}
\usepackage[english]{babel}
\usepackage{fancyhdr}
\usepackage{fancybox}
\usepackage{tikz}

\nonstopmode \numberwithin{equation}{section}
\theoremstyle{plain}
\newtheorem{cor}[equation]{Corollary}
\newtheorem{lem}[equation]{Lemma}

\newtheorem{prop}{Proposition}

\newtheorem{conj}{Conjecture}

\theoremstyle{definition}
\newtheorem{defn}{Definition}[section]

\newtheorem{thm}{Theorem}[section]
\newtheorem{rem}{Remark}[section]
\newtheorem{prob}{Problem}[section]

\newcounter{minutes}
\divide\time by 60
\newcounter{hours}
\multiply\time by 60
\addtocounter{minutes}{-\time}

\newcounter {own}
\def\theown {\thesection       .\arabic{own}}

\newenvironment{pf}[1][]{%
 \vskip 3mm
 \noindent
 \ifthenelse{\equal{#1}{}}%
  {{\slshape Proof. }}%
  {{\slshape #1.} }%
 }%
{\qed\bigskip}

\newcounter{alphabet}

\def\be{\begin{equation}}
\def\ee{\end{equation}}

\newcommand{\bee}{\begin{enumerate}}
\newcommand{\eee}{\end{enumerate}}

\newcommand{\blem}{\begin{lem}}
\newcommand{\elem}{\end{lem}}
\newcommand{\bthm}{\begin{thm}}
\newcommand{\ethm}{\end{thm}}
\newcommand{\bcor}{\begin{cor}}
\newcommand{\ecor}{\end{cor}}
\newcommand{\beg}{\begin{examp}}
\newcommand{\eeg}{\end{examp}}
\newcommand{\begs}{\begin{examples}}
\newcommand{\eegs}{\end{examples}}

\newcommand{\bdefn}{\begin{defn}}
\newcommand{\edefn}{\end{defn}}

\newcommand{\bprob}{\begin{prob}}
\newcommand{\eprob}{\end{prob}}
\newcommand{\bei}{\begin{itemize}}
\newcommand{\eei}{\end{itemize}}

\newcommand{\bcon}{\begin{conj}}
\newcommand{\econ}{\end{conj}}
\newcommand{\bcons}{\begin{conjs}}
\newcommand{\econs}{\end{conjs}}
\newcommand{\bprop}{\begin{prop}}
\newcommand{\eprop}{\end{prop}}
\newcommand{\br}{\begin{rem}}
\newcommand{\er}{\end{rem}}
\newcommand{\brs}{\begin{rems}}
\newcommand{\ers}{\end{rems}}
\newcommand{\bo}{\begin{obser}}
\newcommand{\eo}{\end{obser}}
\newcommand{\bos}{\begin{obsers}}
\newcommand{\eos}{\end{obsers}}
\newcommand{\bpf}{\begin{pf}}
\newcommand{\epf}{\end{pf}}
\newcommand{\ba}{\begin{array}}
\newcommand{\ea}{\end{array}}
\newcommand{\beq}{\begin{eqnarray}}
\newcommand{\beqq}{\begin{eqnarray*}}
\newcommand{\eeq}{\end{eqnarray}}
\newcommand{\eeqq}{\end{eqnarray*}}

\begin{document}

\title{Bohr-type inequalities for classes of analytic maps and K-quasiconformal harmonic mappings}

\author{Molla Basir Ahamed}
\address{Molla Basir Ahamed, Department of Mathematics, Jadavpur University, Kolkata-700032, West Bengal, India.}
\email{mbahamed.math@jadavpuruniversity.in}
\author{Sabir Ahammed}
\address{Sabir Ahammed, Department of Mathematics, Jadavpur University, Kolkata-700032, West Bengal, India.}
\email{sabira.math.rs@jadavpuruniversity.in}

\subjclass[{AMS} Subject Classification:]{Primary 30A10, 30H05, 30C35, Secondary 30C45}

\def\thefootnote{}
\footnotetext{ {\tiny File:~\jobname.tex,
printed: \number\year-\number\month-\number\day,
          \thehours.\ifnum\theminutes<10{0}\fi\theminutes }
} \makeatletter\def\thefootnote{\@arabic\c@footnote}\makeatother

\begin{abstract} 
In this paper, a significant improvement has been achieved in the classical Bohr's inequality for the class $ \mathcal{B} $ of analytic self maps defined on the unit disk $ \mathbb{D} $. More precisely, we generalize and improve several Bohr-type inequalities by combining appropriate improved and refined versions of the classical Bohr's inequality with some methods concerning the area measure of bounded analytic functions in $ \mathcal{B} $. In addition, we obtain Bohr-type and Bohr-Rogosinski-type inequalities for the subordination class and also for the class of $ K $-quasiconformal harmonic mappings. All the results are proved to be sharp.
\end{abstract}

\maketitle
\pagestyle{myheadings}
\markboth{M. B. Ahamed and S. Ahammed}{Bohr-type inequalities for classes of analytic maps and K-quasiconformal harmonic mappings}

\section{Introduction}
Bohr's classical theorem, examined a hundred years ago, has given rise to the Bohr's phenomenon, leading to intense research activity these days. Bohr's seminal work on the power series in complex analysis in the year $ 1914 $ \cite{Bohr-1914} has sparked a flurry of study in complex analysis and related fields. This work is popularly referred to as Bohr’s phenomenon. A number of fascinating breakthroughs have been made in this field in recent years. The generalization of Bohr’s theorem is now an active area of research: Aizenberg \textit{et al.} \cite{Aizenberg-Aytuna-Djakov-2001}, and Aytuna and Djakov \cite{Ayt & Dja & BLMS & 2013} studied the Bohr's property of bases for holomorphic functions; Ali \textit{et al.} \cite{Ali & Abdul & Ng & CVEE & 2016} found the Bohr radius for the class of starlike logharmonic mappings; while Paulsen and Singh \cite{Paulsen-PLMS-2002} extended the Bohr's inequality to Banach algebras. The main goal of this paper is to look into Bohr's phenomenon in the context of bounded analytical functions defined on the unit disk $ \mathbb{D} $, as well as K-quasiconformal harmonic mappings.
\subsection{Classical Bohr's inequality and different aspects of it}
Let $ \mathcal{A} $ denote the set of all analytic functions of the form $ f(z)=\sum_{n=0}^{\infty}a_nz^n $ defined on $ \mathbb{D}:=\{z\in\mathbb{C}:|z|<1\} $ and we define the class $ \mathcal{B}:=\{f\in\mathcal{A} : |f(z)|\leq 1\; \mbox{in}\; \mathbb{D}\} $.  Let us start with a remarkable result of Bohr's \cite{Bohr-1914}, published in $ 1914 $, dealing with a problem connected with Dirichlet series and number theory, which stimulated a lot of research activity into geometric function theory.
\begin{thm}\cite{Bohr-1914}\label{th-1.1}
If $ f(z)=\sum_{n=0}^{\infty}a_nz^n\in\mathcal{B} $, then 
\begin{equation}\label{e-1.2}
M_f(r):=\sum_{n=0}^{\infty}|a_n|r^n\leq 1 \;\; \mbox{for}\;\; |z|=r\leq\frac{1}{3}.
\end{equation}
\end{thm}
The inequality fails when $ r>{1}/{3} $ in the sense that there are functions in $ \mathcal{B} $ for which the inequality is reversed when $ r>{1}/{3} $.  Bohr initially shows the inequality \eqref{e-1.2} for $|z|\leq1/6$. Subsequently, the inequality \eqref{e-1.2}  improved  for $|z|\leq1/3$ by M. Riesz, I. Schur and F. Wiener and showed that the constant $1/3$ is best possible. It is quite natural that the constant $1/3$ and the inequality \eqref{e-1.2} are called respectively, the Bohr radius and the Bohr's inequality for the class $\mathcal{B}$.  Several other proofs of this interesting inequality were given in different articles  (see \cite{Sidon-1927,Paulsen-PLMS-2002,Tomic-1962}).\vspace{1.2mm}

In the majorant series $M_f(r)$, it is worth pointing out that the beginning terms play crucial role to study the Bohr-type inequalities. For instance, Tomic \cite{Tomic-1962} has proved  that if $|a_0|=0$, then the inequality \eqref{e-1.2}  for $r \leq {1}/{2}$ and if the term $|a_0|$ is replaced by  $|a_0|^2$, then the constant $1/3$ could be replaced by $1/2$. In addition, if  $|a_0|$ is replaced by $|f(z)|$, then the constant $1/3$ could be replaced by $\sqrt{5}-2$ which is best possible (see e.g. \cite{Ponnusamy-2017,Kayu-Kham-Ponnu-2021-JMAA}). On other hand Liu \emph{et al.} \cite{Liu-Shang-Xu-JIA-2018} have showed that if the term $ |a_0|+|a_1||z| $ is replaced by $ |f(z)|+|f^{\prime}(z)||z| $ in the majorant series, then the constant $ 1/3 $ could be replaced by the sharp constant $ (\sqrt{17}-3)/4 $. The majorant series belongs to very important class of series with non-negative terms. Another category of series, known as the alternating series, includes the form $A_f(z):=\sum_{k=0}^{\infty}(-1)^k|a_n|r^k$. In $ 2017 $, Ali \emph{et al.} \cite{Ali-2017} established the Bohr’s phenomenon for the classes of even and odd analytic functions and for alternating series showing that $ |A_f(z)|\leq 1 $ on the disk $ \mathbb{D}_{1/\sqrt{3}} $ and the radius $ 1/\sqrt{3} $ is best possible. Recently, Ponnusamy \textit{et el.} \cite{Pon-Shm-Star-JMAA-2024} have estimated for the Bohr radius  in some classes of analytic functions in $ \mathbb{D} $, associated with linearly invariant families of finite order. \vspace{1.2mm} 

We now define what is Bohr inequality for an arbitrary class of functions. 
\begin{defn}
	A class $ \mathcal{F} $ consisting of analytic functions $ f(z)=\sum_{n=0}^{\infty}a_nz^n $ in the unit disk $ \mathbb{D} $ is said to satisfy Bohr's phenomenon if there exists  $ r_f>0 $ such that 
	\begin{align}\label{e-1.1}
		\sum_{n=0}^{\infty}|a_n|r^n\leq 1 \;\; \mbox{for}\;\; |z|=r\leq r_f.
	\end{align}
	The largest radius $ r_f $ is called the Bohr radius and the inequality \eqref{e-1.1} is known as the Bohr's inequality for the class $ \mathcal{F} $.
\end{defn}
 The real burst of activity on Bohr's inequality actually happened in $ 1990 $s after the publication of the articles of Boas and Khavinson \cite{Boas-1997} establishing Bohr's power series theorem in several complex variables. Furthermore, Dixon \cite{Dixon & BLMS & 1995} applied Bohr's classical theorem to construct a Banach algebra that is not classified as an operator algebra but still satisfies the non-unital von Neumann inequality. Subsequently, Paulson and Singh \cite{Paulsen-PAMS-2004}, Blasco \cite{Blasco-2010}, and Bhowmik and Das \cite{Bhowmik-Das-AM-2021} have extended the Bohr's inequality in the context of Banach algebra. In fact, after the Bohr's inequality was generalized from analytic functions in $ \mathbb{C} $ to holomorphic functions in $ \mathbb{C}^n $ by Boas and Khavinson, a variety of results on Bohr’s theorem in higher dimensions appeared in the last three decades. In this context and in other aspects, we refer to \cite{Aizn-PAMS-2000,Allu-CMB-2022,Alkhaleefah-PAMS-2019,Bayart-2014,Bhowmik-2018,Boas-1997,Defant-2011,Lata-Singh-PAMS-2022,Liu-Ponnusamy-PAMS-2021} and the references therein. Various results in Bohr's phenomenon are established (see e.g. \cite{Ahamed-AASFM-2022,Allu-BSM-2021,Beneteau-2004}), and refined versions of Bohr's inequality are proved (see e.g. \cite{Liu-Liu-Ponnusamy-2021}). An open problem raised in \cite{Djakov-JA-20000} has been settled in \cite{Kay & Pon & AASFM & 2019}. In \cite{Ponnusamy-RM-2020,Ponnusamy-HJM-2021,Ponnusamy-JMAA-2022}, the authors have investigated Bohr's inequality for a certain class of power series  replacing $ r^n $ by $ \phi_n(r) ,$  where $ \phi_n(r) $ is a non-negative continuous function on $ [0,1) $ such that $ \sum_{n=0}^{\infty}\phi_n(r) $ converges locally uniformly with respect to $ r\in [0, 1) $. Recently, Ahamed \emph{et al.} \cite{Ahamed-AASFM-2022}, and Evdoridis \emph{et al.} \cite{Ponnusamy-RM-2021} have studied the Bohr's phenomenon for analytic functions on shifted disks, whereas Allu and Halder \cite{Allu-JMAA-2021} have established Bohr-type inequalities for certain classes of starlike and convex univalent functions in $ \mathbb{D}$. In the past few years, Bohr’s inequality has become the subject of extensive investigation and the inequality has been extended in numerous directions and settings, the interested reader can refer to (\cite{Aha-Aha-MJM-2023}, \cite{Boas-1997,Boas-2000}, \cite{Aizn-PAMS-2000}-\cite{Aizeberg-PLMS-2001}, \cite{Kumar-PAMS-2023} and \cite{Paulsen-PLMS-2002}-\cite{Paulsen-BLMS-2006}). \vspace{1.2mm}
  
  The Bohr inequality was studied by Kayumov and Ponnusamy \cite{Kayumov-CMFT-2017} in the context of two classes of analytic functions: those that are subordinate to univalent functions and those that are subordinate to odd univalent functions. Liu \cite{Liu-JMAA-2021} subsequently investigated the Bohr-type inequalities for various subordination classes, using the Schwarz function $\omega(z)$ in the place of $z$ as in the classical Bohr's inequality. To gain insight into the recent developments of Bohr's inequality, we suggest consulting \cite{Ahamed-RMJM-2021,Allu-Halder-Banach,Allu-IM-2021,Das-JMAA-2022,Huang-Liu-Ponnu-CVEE-2021,Kumar-Sahoo-MJM-2021}.
\subsection{The Bohr and Rogosinski inequalities for bounded analytic functions}
Similar to Bohr's radius, there is also the concept of Rogosinski radius, however, this is less known than Bohr radius. According to Rogosinski, the radius is as follows: 
\begin{defn} (Rogosinski radius)
	Let $f(z)=\sum_{n=0}^{\infty} a_{n}z^{n}$ be analytic in $\mathbb{D}$ and the corresponding partial sum of $f$ be defined by $S_{N}(z):=\sum_{n=0}^{N-1} a_{n}z^{n}$. Then, for every $N \in\mathbb{N}$, 
	\begin{align*}
		|S_{N}(z)|\leq1\;\; \mbox{for}\;\; |z|\leq \frac{1}{2}.
	\end{align*} The radius $1/2$ is sharp and is known as the Rogosinski radius.
\end{defn}
 The Rogosinski radius served as the foundation for the Bohr-Rogosinski sum $R_{N}^{f}(z)$, as established by Kayumov and Ponnusamy \cite{Ponnusamy-2017}
\begin{equation}
	R_{N}^{f}(z):=|f(z)|+ \sum_{n=N}^{\infty} |a_{n}||z|^{n}.
\end{equation}
Here, the Bohr-Rogosinski radius is the largest number $ r_N>0 $ such that  $ R_{N}^{f}(z)\leq 1 $ for $ |z|\leq  r_N $. It is worth mentioning that $|S_{N}(z)|=\big|f(z)-\sum_{n=N}^{\infty} a_{n}z^{n}\big| \leq |R_{N}^{f}(z)|$. In this context, it is pertinent to emphasize that for $ N=1 $, similarly, $R_N^f(z)$, which is related to the classical Bohr's sum (i.e. majorant series) when $|a_0|=|f(0)|$ is replaced by $|f(z)|$, gives Rogosinski radius in the case of bounded analytic functions in $\mathbb{D}$. In relation to the notions of Rogosinski's inequality and Rogosinski's radius explored in \cite{Rogosinski-1923,Schur-1925}, Kayumov and Ponnusamy \cite{Ponnusamy-2017} (see also \cite{Kayu-Kham-Ponnu-2021-JMAA}) established the following result regarding the Bohr–Rogosinski radius for analytic functions in $ \mathcal{B} $.
\begin{thm}\cite{Ponnusamy-2017}\label{th-1.5}
	Suppose that  $  f(z)=\sum_{n=0}^{\infty}a_nz^n \in\mathcal{B} $ . Then 
\begin{equation*}
	 |f(z)|+\sum_{n=N}^{\infty}|a_n|r^n\leq 1\;\;\mbox{for}\;\; r\leq R_N,
\end{equation*}
where $ R_N $ is the positive root of the equation $ 2(1+r)r^N-(1-r)^2=0 $. The radius $ R_N $ is the best possible. Moreover, 
\begin{equation*}
	|f(z)|^2+\sum_{n=N}^{\infty}|a_n|r^n\leq 1\;\;\mbox{for}\;\; r\leq R^{\prime}_N,
\end{equation*}
where $ R^{\prime}_N $ is the positive root of the equation $ (1+r)r^N-(1-r)^2=0 $. The radius $ R^{\prime}_N $ is the best possible.
\end{thm}
A number of improvements and extensions have been made to Theorem \ref{th-1.5} specifically for functions in the class $\mathcal{B}$ (see e.g. \cite{Aizenberg-AMP-2012,Das-JMAA-2022,Ismagilov-2020-JMAA,Liu-Liu-Ponnusamy-2021} and references therein).  
\section{Improved and refined versions of classical Bohr's inequality for the class $ \mathcal{B} $}
 Let $f$ be holomorphic in $\mathbb D$, and for $0<r<1$,  let $\mathbb D_r:=\{z\in \mathbb C: |z|<r\}$.
Throughout the paper,  $S_r:=S_r(f)$ denotes the planar integral

$$S_r=\int_{\mathbb D_r} |f'(w)|^2 d A(w).$$

Note that if  $f(z)=\sum_{n=0}^\infty a_nz^n$, than $$S_r=\pi \sum_{n=1}^\infty n|a_n|^2 r^{2n}.$$  If $f$ is a univalent function, then $S_r$ is the area of  $f(\mathbb D_r)$. 
\vspace{1.2mm}
Kayumov \emph{et al.} \cite{Kayumov-CRACAD-2018} proved  the following sharp inequality  for functions in the class 
$ \mathcal{B}$,
\begin{align}\label{e-1.55}
	\frac{S_r}{\pi}=\sum_{n=1}^{\infty}n|a_n|^2r^{2n}\leq r^2\frac{(1-|a_0|^2)^2}{(1-|a_0|^2r^2)^2} \,\,\mbox{for}\,\, 0<r\leq1/\sqrt{2}.
\end{align}
\vspace{1.2mm}
\par In recent times, the quantity $S_r$ has been widely used as a tool for investigating the Bohr's phenomenon in the setting of analytic functions (see e.g. \cite{Ahamed-AASFM-2022,Huang-Liu-Ponnu-CVEE-2021}), harmonic mappings (see e.g. \cite{Ahamed-AMP-2021,Ahamed-CVEE-2021,Ahamed-RMJM-2021}) and also for operator valued functions in the study of the multidimensional Bohr's inequality (see \cite{Allu-CMB-2022}). In 2018, Kayumov \emph{et. al.} \cite{Kayumov-CRACAD-2018} proved the following improved version of the Bohr's inequality in terms of $ S_r $.
\begin{thm}\cite{Kayumov-CRACAD-2018}\label{th-1.12}
Suppose that  $  f(z)=\sum_{n=0}^{\infty}a_nz^n \in\mathcal{B}. $  Then 
	\begin{equation*}
		\sum_{n=0}^{\infty}|a_n|r^n+ \frac{16}{9} \left(\frac{S_{r}}{\pi}\right) \leq 1 \quad \mbox{for} \quad r \leq \frac{1}{3}
	\end{equation*}
	and the numbers $1/3$, $16/9$ cannot be improved. Moreover, 
	\begin{equation*}
		|a_{0}|^{2}+\sum_{n=1}^{\infty}|a_n|r^n+ \frac{9}{8} \left(\frac{S_{r}}{\pi}\right) \leq 1 \quad \mbox{for} \quad r \leq \frac{1}{2}
	\end{equation*}
	and the numbers $1/2$, $9/8$ cannot be improved.
\end{thm}
In 2020, Ismagilov \emph{et al.} \cite{Ismagilov-2020-JMAA} further investigated Theorem \ref{th-1.12} and proved the following sharp inequalities by incorporating the non-linearity of the quantity $ S_r/\pi $.
\begin{thm}\cite{Ismagilov-2020-JMAA} \label{th-1.8}
	Suppose that $ f(z)=\sum_{n=0}^{\infty}a_nz^n\in\mathcal{B} $. Then
	\begin{equation*}
		\sum_{n=0}^{\infty}|a_n|r^n+\frac{16}{9}\left(\frac{S_r}{\pi}\right)+\lambda\left(\frac{S_r}{\pi}\right)^2\leq 1\;\; \mbox{for}\;\; r\leq\frac{1}{3},
	\end{equation*}
where
\begin{equation*}
	\lambda=\frac{4(486-261a-324a^2+2a^3+30a^4+3a^5)}{81(1+a)^3(3-5a)}=18.6095...
\end{equation*}
and $ a\approx 0.567284 $ is the unique root of the equation $-405+473t+402t^2+38t^3+3t^4+t^5=0 $ in the interval $ (0,1) $.
The equality is achieved for the function $ f_a(z):=(a-z)/(1-az) $.
\end{thm}
\begin{thm}\cite{Ismagilov-2020-JMAA}\label{th-1.10}
	Suppose that $ f(z)=\sum_{n=0}^{\infty}a_nz^n\in\mathcal{B} $. Then 
	\begin{equation*}
		|f(z)|^2+\sum_{n=1}^{\infty}|a_n|r^n+\frac{16}{9}\left(\frac{S_r}{\pi}\right)+\lambda\left(\frac{S_r}{\pi}\right)^2\leq 1\;\; \mbox{for}\;\; r\leq\frac{1}{3},
	\end{equation*}
	where
	\begin{equation*}
		\lambda=\frac{-81+1044a+54a^2-116a^3-5a^4}{162(a+1)^2(2a-1)}=16.4618...
	\end{equation*}
	and $ a\approx 0.537869 $, is the unique root of the equation $-513+910t+80t^2+2t^3+t^4=0$
in the interval $ (0,1) $.
	The equality is achieved for the function $ f_a $.
\end{thm}
 In \cite{Ismagilov-2021-JMS}, Ismagilov \emph{et al.} have observed that
\begin{align*}
	1-\dfrac{S_r}{\pi}\geq \dfrac{(1-r^2)(1-r^2|a_0|^4)}{(1-|a_0|^2r^2)^2},
\end{align*}
and hence, in view of \eqref{e-1.55}, they established that
\begin{align}\label{e-11.18}
	\dfrac{S_r}{\pi-S_r}\leq \dfrac{r^2(1-|a_0|^2)^2}{(1-r^2)(1-r^2|a_0|^4)}.
\end{align}
With the assistance of this new configuration, Ismagilov \emph{et al.} \cite{Ismagilov-2021-JMS} improved Theorem \ref{th-1.12} and proved the following sharp inequalities in terms of $S_r/(\pi-S_r)$ in the place of $ S_r/\pi $ keeping all others conditions intact.

\begin{thm} \label{th-1.15} \cite{Ismagilov-2021-JMS}
	Suppose that $ f\in\mathcal{B} $ such that $ f(z)=\sum_{n=0}^{\infty}a_nz^n $. Then 
	\begin{align*}
		\sum_{n=0}^{\infty}|a_n|r^n+ \frac{16}{9} \left(\frac{S_{r}}{\pi-S_r}\right) \leq 1 \quad \mbox{for} \quad r \leq \frac{1}{3},
	\end{align*}
	and the number $16/9$ cannot be improved. Moreover, 
	\begin{align*}
		|a_{0}|^{2}+\sum_{n=1}^{\infty}|a_n|r^n+ \frac{9}{8} \left(\frac{S_{r}}{\pi-S_r}\right) \leq 1 \quad \mbox{for} \quad r \leq \frac{1}{2},
	\end{align*}
	and the number $9/8$ cannot be improved.
\end{thm}
Recently, there has been a good deal of research on Bohr’s phenomenon in various settings, including a refined formulation of the classical version of the Bohr's inequality. In the following, we recall further results that are an improved or refined version of the classical Bohr's inequality. The subsequent results, which are some refined version of the classical Bohr's inequality, have been established by Liu \textit{et al.} in \cite{Liu-Liu-Ponnusamy-2021}.
\begin{thm}\label{th-1.16} \cite{Liu-Liu-Ponnusamy-2021}
	Suppose that $ f(z)=\sum_{n=0}^{\infty}a_nz^n\in\mathcal{B} .$  Then
	\begin{align*}
		\sum_{n=0}^{\infty}|a_n|r^n+\left(\frac{1}{1+|a_0|}+\frac{r}{1-r}\right)\sum_{n=1}^{\infty}|a_n|^2r^{2n}+\frac{8}{9}\left(\frac{S_r}{\pi}\right)\leq 1 \quad \mbox{for} \quad r \leq \frac{1}{3}
	\end{align*}
	and the constants $8/9$ and $ 1/3 $ cannot be improved. Moreover,
	\begin{align*}
	|a_0|^2+\sum_{n=1}^{\infty}|a_n|r^n+\left(\frac{1}{1+|a_0|}+\frac{r}{1-r}\right)\sum_{n=1}^{\infty}|a_n|^2r^{2n}+\frac{9}{8}\left(\frac{S_r}{\pi}\right)\leq 1
\end{align*}
for $ r\leq {1}/{(3-|a_0|)} $ and the constant $ 9/8 $ cannot be improved.
\end{thm}
\begin{thm}\label{th-1.17} \cite{Liu-Liu-Ponnusamy-2021}
		Suppose that $ f(z)=\sum_{n=0}^{\infty}a_nz^n\in\mathcal{B} $. Then
		\begin{align*}
		\sum_{n=0}^{\infty}|a_n|r^n+\left(\frac{1}{1+|a_0|}+\frac{r}{1-r}\right)\sum_{n=1}^{\infty}|a_n|^2r^{2n}+|f(z)-a_0|\leq 1 \quad \mbox{for} \quad r \leq \frac{1}{5}
	\end{align*}
	and the number $1/5$ cannot be improved. Moreover,
\begin{align*}
|a_0|^2+\sum_{n=1}^{\infty}|a_n|r^n+\left(\frac{1}{1+|a_0|}+\frac{r}{1-r}\right)\sum_{n=1}^{\infty}|a_n|^2r^{2n}+|f(z)-a_0|\leq 1 
\end{align*}
for $r\leq1/3$ and the number $1/3$ cannot be improved.
\end{thm}
Huang \emph{et al.} \cite{Huang-Liu-Ponnu-CVEE-2021} addressed the following problem of further studying Theorems \ref{th-1.8} and \ref{th-1.16} in their recent study.
\begin{prob}\label{q-1.18}
	Can we derive a more accurate formulation of Theorems \ref{th-1.16} in the context of Theorems \ref{th-1.8} and \ref{th-1.10} while retaining the radius and including a non-negative term?
\end{prob}
As a matter of fact, Huang \emph{et al.} \cite{Huang-Liu-Ponnu-CVEE-2021}  have demonstrated the following two results, which solved the Problem \ref{q-1.18}.
\begin{thm}\label{th-1.20}\cite{Huang-Liu-Ponnu-CVEE-2021}
	Suppose that $ f(z)=\sum_{n=0}^{\infty}a_nz^n\in\mathcal{B} $. Then 
	\begin{align*}
		\sum_{n=0}^{\infty}|a_n|r^n+\left(\frac{1}{1+|a_0|}+\frac{r}{1-r}\right)\sum_{n=1}^{\infty}|a_n|^2r^{2n}+\frac{8}{9}\left(\frac{S_r}{\pi}\right)\\\quad+\lambda\left(\dfrac{S_r}{\pi}\right)^2\leq 1 \;\; \mbox{for}\;\; r\leq\dfrac{1}{3},
	\end{align*}
		where 
	\begin{equation*}
		\lambda=\dfrac{-2673+2502a+2025a^2-332a^3-255a^4+6a^5+7a^6}{162(1+a)^3(5a-3)}=14.796883...
	\end{equation*}
	and $a\approx0.587459$ is the unique root of the equation  
	\begin{align*}
		2t^6+4t^5+14t^4-40t^3-1218t^2-1756t+1458=0
	\end{align*} 
in the interval $(0,1)$.
The equality  is achieved for the function 
	$f_a$.
\end{thm}
	\begin{thm}\cite{Huang-Liu-Ponnu-CVEE-2021}\label{th-1.21}
			Suppose that $ f(z)=\sum_{n=0}^{\infty}a_nz^n\in\mathcal{B} $ . Then
	\begin{align*}
		|a_0|^2+\sum_{n=1}^{\infty}|a_n|r^n+\left(\frac{1}{1+|a_0|}+\frac{r}{1-r}\right)\sum_{n=1}^{\infty}|a_n|^2r^{2n}+\frac{9}{8}\left(\frac{S_r}{\pi}\right)\\\quad+\lambda\left(\dfrac{S_r}{\pi}\right)^2\leq 1 \;\; \mbox{for}\;\; r\leq\dfrac{1}{3-|a_0|},
	\end{align*}
 where 
	\begin{align*}
		\lambda&=\dfrac{-80919+119556a-57591a^2+11664a^3-1620a^4}{8(a-3)^2(1+a)^2(9a^3-33a^2+29a-1)}=13.966088...
	\end{align*}
	and $a\approx0.638302$ is the unique root of the equation  
	\begin{align*}
		-1296t^8+17172t^7-154386t^6+798660t^5-2361960t^4+4132944t^3\\\quad-4244238t^2+2344464t-524880=0
	\end{align*}
 in the interval $(0,1)$.
	The equality is achieved for the function $ f_a $.
\end{thm}
In view of Theorem \ref{th-1.17}, it is natural to investigate the following problem for its further improvements.
\begin{prob}\label{q-1.21}
	Is it possible to derive a sharp version of Theorem \ref{th-1.17} with an additional non-negative term and without reducing the radius?
\end{prob}
From the discussions above, it is evident that several authors have established improved and refined versions of Bohr's inequality using the quantity $S_r/\pi$. However, none of them have attempted to establish results for theorems mentioned above in terms of $S_r/(\pi-S_r)$, except for Theorem \ref{th-1.15}.  The aim of this paper is to address the following problem, based on the motivation provided by Theorem \ref{th-1.15} and taking into consideration Theorems \ref{th-1.20} and \ref{th-1.21}.
\begin{prob}\label{q-1.25}
	Can we derive a sharp version of Theorems \ref{th-1.20} and \ref{th-1.21} by replacing the quantity $ S_r/\pi $ with $ S_r/(\pi-S_r) $ without altering the radius?
\end{prob}
In this section, we will establish precise Bohr-type inequalities for Theorems \ref{th-1.17}, \ref{th-1.20} and \ref{th-1.21} to address the Problems \ref{q-1.21} and \ref{q-1.25}. As a matter of fact, we are interest in further investigations on the Bohr's phenomenon for the class $ \mathcal{B} $ with the setting $ S_r/(\pi-S_r) $.\vspace{1.2mm} 

The subsequent statements represent the primary results of this section.
\begin{thm}\label{th-2.2}
	Suppose that $ f(z)=\sum_{n=0}^{\infty}a_nz^n\in\mathcal{B} $. Then
\begin{align}\label{e-2.2}
	\mathcal{A}_{1,f}(r):&=\sum_{n=0}^{\infty}|a_n|r^n+\left(\frac{1}{1+|a_0|}+\frac{r}{1-r}\right)\sum_{n=1}^{\infty}|a_n|^2r^{2n}+|f(z)-a_0|\\&\nonumber+\quad\dfrac{18}{5}\left(\dfrac{S_r}{\pi}\right)+\lambda\left(\dfrac{S_r}{\pi}\right)^2\leq 1 \;\; \mbox{for}\;\;r\leq \dfrac{1}{5},
\end{align}
where 
\begin{align*}
	\lambda=\dfrac{(25-\alpha^2)(11125-3810\alpha-2300\alpha^2-30\alpha^3+7\alpha^4)}{2500(1+\alpha)^3(3-5\alpha)}=118.383318...
\end{align*}
and $\alpha\approx 0.564084$ is the unique root in the interval $(0,1)$ of the equation 
\begin{align*}
	32500-44845t-22435t^2-370t^3+14t^4-t^5+t^6=0.
\end{align*} 
The equality  is achieved for the function
$ f_a(z)=(a-z)/(1-az) $.
\end{thm}

\begin{thm}\label{th-2.18}
	Suppose that $ f(z)=\sum_{n=0}^{\infty}a_nz^n\in\mathcal{B} $. Then
	\begin{align}\label{e-2.19}
		\mathcal{A}_{2,f}(r):&=\sum_{n=0}^{\infty}|a_n|r^n+\left(\frac{1}{1+|a_0|}+\frac{r}{1-r}\right)\sum_{n=1}^{\infty}|a_n|^2r^{2n}+\frac{8}{9}\left(\frac{S_r}{\pi-S_r}\right)\\&\nonumber\quad+\lambda\left(\dfrac{S_r}{\pi-S_r}\right)^2\leq 1 \;\; \mbox{for}\;\; r\leq\dfrac{1}{3},
	\end{align}
 where 
	\begin{align*}
		\lambda=\dfrac{32(27+18\beta+27\beta^2-28\beta^3-15\beta^4-6\beta^5-7\beta^6)}{81(1+\beta)^3(3-5\beta)}=12.342793...
	\end{align*}
	and $\beta\approx 0.531615$ is the unique root in the interval $ (0,1) $ of the equation
	\begin{align*}
		81-126t-54t^2+14t^3-12t^4+2t^5+2t^6-2t^7-t^8=0.
	\end{align*}
	The equality is achieved for the function $ f_a(z)=(a-z)/(1-az) $.
\end{thm}
\begin{thm}\label{th-2.21}
	Suppose that $ f(z)=\sum_{n=0}^{\infty}a_nz^n\in\mathcal{B} $. Then
	\begin{align}\label{e-2.18}
		\mathcal{A}_{3,f}(r):&=|a_0|^2+\sum_{n=1}^{\infty}|a_n|r^n+\left(\frac{1}{1+|a_0|}+\frac{r}{1-r}\right)\sum_{n=1}^{\infty}|a_n|^2r^{2n}\\&\nonumber\quad+\frac{9}{8}\left(\frac{S_r}{\pi-S_r}\right)+\lambda\left(\dfrac{S_r}{\pi-S_r}\right)^2\leq 1 \;\; \mbox{for}\;\; r\leq\dfrac{1}{3-|a_0|},
	\end{align}
 where 
	\begin{align*}
		\lambda&=\dfrac{-13581+22491\gamma-14085\gamma^2-2584\gamma^3+8565\gamma^4-5130\gamma^5+903\gamma^6+636\gamma^7-324\gamma^8+40\gamma^9}{8(3-\gamma)^3(1+\gamma)^2(1-7\gamma+4\gamma^2)}\\&\nonumber=10.787939...
	\end{align*}
	and $\gamma\approx0.571317$ is the unique root of the equation $55647-212544t+296244t^2-200754t^3+61377t^4+5198t^5-10420t^6+972t^7+960t^8+408t^9-420t^{10}+100t^{11}-8t^{12}=0$ in the interval $(0,1)$.
	The equality is achieved for the function
	$f_a(z)=(a-z)/(1-az)$.
\end{thm}
Before we delve into the main results of this section, we lay out some necessary lemmas that will be useful in proving them.
\begin{lem} \label{lem-3.2}\cite{Ponnusamy-2017}
	Suppose that $f(z)=\sum_{n=0}^{\infty}a_nz^n\in\mathcal{B}$. Then the following sharp inequality holds:
	\begin{align} \label{e-3.3}
		\frac{S_r}{\pi}=\sum_{n=1}^{\infty}n|a_n|^2r^{2n}\leq r^2\frac{(1-|a_0|^2)^2}{(1-|a_0|^2r^2)^2}\; \,\,\mbox{for}\;\; 0<r\leq1/\sqrt{2}.
	\end{align}
\end{lem}
\begin{lem} \label{lem-3.4} \cite{Kayumov-CMFT-2017, Kayumov-2018-JMAA}
	Suppose that $f(z)=\sum_{n=0}^{\infty}a_nz^n\in\mathcal{B}$. Then we have 
	\begin{align}\label{e-3.5}
		\sum_{n=1}^{\infty}|a_n|r^n\leq
		\begin{cases}
			M(r):=r\displaystyle\frac{1-|a_0|^2}{1-r|a_0|}\quad \mbox{for}\;|a_0|\geq r,\vspace{2mm}\\
			N(r):=r\displaystyle\frac{\sqrt{1-|a_0|^2}}{\sqrt{1-r^2}}\quad\mbox{for}\; |a_0|<r.
		\end{cases}
	\end{align}
The largest integer which is not greater than the given real number $x$ is denoted by $\lfloor{x}\rfloor$.
\end{lem}
\begin{lem}\cite{Liu-Liu-Ponnusamy-2021} \label{lem-3.5} 
	Suppose that $f(z)=\sum_{n=0}^{\infty}a_nz^n\in\mathcal{B}$. Then for any $N\in\mathbb{N}$, the following inequality holds:
	\begin{align*}
		\sum_{n=N}^{\infty}|a_n|r^n+sgn(t)\sum_{n=1}^{t}|a_n|^2\dfrac{r^N}{1-r}+\left(\dfrac{1}{1+|a_0|}+\dfrac{r}{1-r}\right)\sum_{n=t+1}^{\infty}|a_n|^2r^{2n}\leq \dfrac{(1-|a_0|^2)r^N}{1-r}, 
	\end{align*}
	\;\;\mbox{for}\;\; $ r\in[0,1),$ where $t=\lfloor{(N-1)/2}\rfloor$.
\end{lem}
\begin{proof}[\bf Proof of Theorem \ref{th-2.2}]
		Let us assume that the function $\mathcal{A}_{1,f}(r)$ is given by \eqref{e-2.2}. By an easy computation, it can be shown that $\mathcal{A}_{1,f}(r)$ is an increasing function of $r$, hence it suffices to prove the inequality \eqref{e-2.2} for $r=1/5$. Moreover, for $r=1/5$, a simple computation using Lemmas \ref{lem-3.2} and \ref{lem-3.4} shows that 		
			\begin{align*}
			\frac{S_{1/5}}{\pi}=\sum_{n=1}^{\infty}n|a_n|^25^{-2n}\leq 25\frac{(1-|a_0|^2)^2}{(25-|a_0|^2)^2}
		\end{align*}
		and  
			\begin{align*}
			\sum_{n=1}^{\infty}|a_n|5^{-n}\leq
			\begin{cases}
				M(1/5)=\displaystyle\frac{1-|a_0|^2}{5-|a_0|}\quad \mbox{for}\;|a_0|\geq 1/5,\vspace{2mm}\\
				N(1/5)=\displaystyle\frac{\sqrt{1-|a_0|^2}}{\sqrt{24}}\quad\mbox{for}\; |a_0|<1/5.
			\end{cases}
		\end{align*} 
			
		First we consider the case $|a_0|\geq1/5$. Based on Lemma \ref{lem-3.5} (with $N=1$), we obtain 
		\begin{align*}
			 \mathcal{A}_{1,f}(1/5)&\leq|a_0|+\dfrac{1-|a_0|^2}{4}+M(1/5)+\dfrac{18}{5}\left(\dfrac{S_{1/5}}{\pi}\right)+\lambda\left(\dfrac{S_{1/5}}{\pi}\right)^2\\&\leq|a_0|+\dfrac{1-|a_0|^2}{4}+\displaystyle\frac{1-|a_0|^2}{5-|a_0|}+90\frac{(1-|a_0|^2)^2}{(25-|a_0|^2)^2}+\lambda\left( 25\frac{(1-|a_0|^2)^2}{(25-|a_0|^2)^2}\right)^2\\&=1-\dfrac{(1-|a_0|)^3}{4(25-|a_0|^2)^4}\Phi(|a_0|),
		\end{align*}
		where
		\begin{align*}
			\Phi(t):=(25-t^2)^2(1015+445t+5t^2-t^3)-2500\lambda(1-t)(1+t)^4.
		\end{align*}
	A simple calculation demonstrates that $\Phi(t)$ has a single stationary point $\alpha=0.564084...$ in the interval $(0,1)$ which is in fact the positive root of the equation $\Phi^\prime(t)=0$.
	We note that $\Phi^\prime(t)=0$ is equivalent to 
	\begin{align*}
	(25-t^2)(11125-3810t-2300t^2-30t^3+7t^4)-2500\lambda(1+t)^3(3-5t)=0
	\end{align*}
which yields the value of $\lambda$ mentioned in the  statement of the theorem. Now we are to show that $\Phi(\alpha)=0$.
	It is easy to see that the equation $\Phi^\prime(\alpha)=0$ is fulfilled automatically (in fact, $\lambda$ was chosen in this manner from the start). By simply plugging the value of $\lambda$ into the expression of $\Phi(t)$, we obtain 
	\begin{align*}
		\Phi(t)=\dfrac{2(25-t^2)(32500-44845t-22435t^2-370t^3+14t^4-t^5+t^6)}{(3-5t)}.
	\end{align*} 
It can be easily shown that $\alpha$ is the unique root of the equation $\Phi(t)=0$ in the interval $[1/5,1]$, hence we must have $\Phi(\alpha)=0$. We observe that $\Phi(1/5)>0$ and $\Phi(1)>0$. Thus, $\Phi(t)\geq 0$ in the interval $[1/5,1]$ which confirms that $\mathcal{A}_{1,f}(r)\leq1$ for $1/5\leq|a_0|\leq1$ and $r\leq1/5$. \vspace{1.2mm} 
 \par Next we consider the case $|a_0|<1/5$. In this case, it is easy to obtain 
	\begin{align*}
		 \mathcal{A}_{1,f}(1/5)&\leq|a_0|+\dfrac{1-|a_0|^2}{4}+\displaystyle\frac{\sqrt{1-|a_0|^2}}{\sqrt{24}}+90 \frac{(1-|a_0|^2)^2}{(25-|a_0|^2)^2}+\lambda\left( 25\frac{(1-|a_0|^2)^2}{(25-|a_0|^2)^2}\right)^2\\&\leq|a_0|+\dfrac{1-|a_0|^2}{4}+\displaystyle\frac{1}{\sqrt{24}}+90 \frac{(1-|a_0|^2)^2}{(25-|a_0|^2)^2}+\lambda\left( 25\frac{(1-|a_0|^2)^2}{(25-|a_0|^2)^2}\right)^2\\&:=\mathcal{A}^*_{1,f}(|a_0|).
		\end{align*}
A straightforward computations show that $\mathcal{A}^*_{1,f}(t)$ is an increasing function of $t$ in the interval $[0,1/5)$ and therefore, the maximum value of $\mathcal{A}^*_{1,f}(t)$ is achieved at $t=1/5$. The maximum value of $\mathcal{A}^*_{1,f}(t)$ is $0.989215...$, which is less than $1$. This proves that $\mathcal{A}_{1,f}(r)\leq1$ for $0\leq|a_0|<1/5$ and $r\leq1/5$.\vspace{1.2mm}

\par In order to show that the constant $\lambda$ is sharp, we consider the function 
\begin{align}\label{e-5.1}
	f_a(z)=\dfrac{a-z}{1-az}=a-(1-a^2)\sum_{n=1}^{\infty}a^{n-1}z^n, z\in\mathbb{D},
\end{align}
where $a\in(0,1)$. For this $f_a$, $a_0=a,$ $a_n=-(1-a^2)a^{n-1}$ for $n\geq 1,$ and we see that  $|f_a(z)-a_0|={(1-a^2)r}/{(1-ar)}$
and hence, 
\begin{align*}
\mathcal{A}_{1,f_a}(r)=a+\dfrac{(1-a^2)r}{1-r}+\dfrac{(1-a^2)r}{1-ar}+\dfrac{18(1-a^2)^2r^2}{5(1-a^2r^2)^2}+\lambda\left(\dfrac{(1-a^2)^2r^2}{(1-a^2r^2)^2}\right)^2.	
\end{align*}

 For $r=1/5,$ a simple computation shows that
\begin{align*}
	\mathcal{A}_{1,f_a}(1/5)&=1+G_{\lambda}(a)+(\lambda_1-\lambda)\left( 25\frac{(1-a^2)^2}{(25-a^2)^2}\right)^2,
	\end{align*}
where 
\begin{align*}
	G_{\lambda}(a):=-1+a+\dfrac{(1-a^2)}{4}+\dfrac{(1-a^2)}{(5-a)}+90 \frac{(1-a^2)^2}{(25-a^2)^2}+\lambda\left( 25\frac{(1-a^2)^2}{(25-a^2)^2}\right)^2.
\end{align*}
Choose $a=\alpha$, where $\alpha$ is unique root of the equation $\Phi(t)=0$ in $(0,1).$ Consequently, $G_{\lambda}(\alpha)=0$ and hence,
$\mathcal{A}_{1,f}(1/5)=1+K(\lambda_1-\lambda),$ where $K>0$. Therefore, we see that $\mathcal{A}_{1,f}(1/5)>1$ in case  $\lambda_1>\lambda$, which proves the sharpness of the assertion and the proof is complete.
\end{proof}
\begin{proof}[\bf Proof of Theorem \ref{th-2.18}] 
	We assume that $\mathcal{A}_{2,f}(r)$ is given by \eqref{e-2.19}. Since $\mathcal{A}_{2,f}(r)$ is an increasing function of $r$, it suffices to prove the inequality \eqref{e-2.19} for $r=1/3$. Using Lemma \ref{lem-3.5} (with $N=1$) for $r=1/3$, we obtain 
	\begin{align*}
		 \mathcal{A}_{2,f}(1/3)&\leq|a_0|+\dfrac{(1-|a_0|^2)}{2}+\dfrac{(1-|a_0|^2)^2}{(9-|a_0|^4)}+\lambda\left(\dfrac{9(1-|a_0|^2)^2}{8(9-|a_0|^4)}\right)^2\\&=1-\dfrac{(1-|a_0|)^3\Psi(|a_0|)}{64(9-|a_0|^4)^2},
	\end{align*}
	where
	\begin{align*}
		\Psi(t):=32(7+3t+t^2+t^3)(9-t^4)-81\lambda(1-t)(1+t)^4.
	\end{align*}
	It is easy to show that $\Psi(t)$ has exactly one stationary point $\beta=0.531615...$, which is the unique root of the equation $\Psi^\prime(t)=0$ in $[0,1]$. It is easy to see that $\Psi^\prime(t)=0$ is equivalent to 
	\begin{align*}
		32(27+18t+27t^2-28t^3-15t^4-6t^5-7t^6)-81\lambda(3+4t-6t^2-12t^3-5t^4)=0,
	\end{align*}
	from which we obtain the value of $\lambda$ mentioned in the statement of the theorem. We aim to show that  $\Psi(\beta)=0$. The equation $\Psi^\prime(\beta)=0$ is fulfilled (in fact, $\lambda$ was chosen in this manner from the start).
	Substituting the value of $\lambda$ in the expression of $\Psi(t)$. This gives 
	\begin{align*}
		\Psi(t)=\dfrac{64(81-126t-54t^2+14t^3-12t^4+2t^5+2t^6-2t^7-t^8)}{3-5t}.
	\end{align*}
It is easy to see that $\beta=0.531615$ is the unique root of the equation $\Psi(t)=0$ in $[0,1]$. Besides this observation, we have $\Psi(0)>0$ and $\Psi(1)>0$. Therefore, $\Psi(t)\geq 0$ in  $[0,1]$ which proves that $\mathcal{A}_{2,f}(r)\leq 1$ for $0\leq|a_0|\leq1$ and $r\leq1/3$. \vspace{1.2mm}
	In order to show that the constant $\lambda$ is sharp, we consider the function $ f_a $ given by \eqref{e-5.1}. For this $f_a$ and $r=1/3$, we easily obtain
	\begin{align*}
		\mathcal{A}_{2,f_a}(1/3)&=1+H_{ \lambda}(a)+(\lambda_2-\lambda)\left(\dfrac{9(1-a^2)^2}{8(9-a^4)}\right)^2,
	\end{align*}
	where \begin{align*}
		H_{ \lambda}(a):=-1+a+\dfrac{1-a^2}{2}+\dfrac{(1-a^2)^2}{(9-a^4)}+\lambda\left(\dfrac{9(1-a^2)^2}{8(9-a^4)}\right)^2.
	\end{align*}
	Choose $a=\beta$, where $ \beta$ is the unique root of the equation $\Psi(t)=0$ in $(0,1)$. Aa a consequence, we see that $H_{\lambda}(\beta)=0$ and hence $ \mathcal{A}_{2,f}(1/3)=1+d(\lambda_1-\lambda), $ where $d>0$. Therefore, $\mathcal{A}_{2,f_a}(1/3)>1$ in case  $\lambda_2>\lambda$. This proves the sharpness assertion and the proof is completed.
	\end{proof}
\begin{proof}[\bf Proof of Theorem \ref{th-2.21}]
		Let $\mathcal{A}_{3,f}(r)$ be given by \eqref{e-2.18}. Since $\mathcal{A}_{3,f}(r)$ is an increasing function of $r$, it is enough to prove the inequality \eqref{e-2.18} for $r=1/(3-|a_0|)$. In view of Lemma \ref{lem-3.5} (with $N=1$), a simple computation shows that
	\begin{align*}
		 \mathcal{A}_{3,f}(1/(3-|a_0|))&\leq|a_0|^2+\dfrac{1-|a_0|^2}{2-|a_0|}+\dfrac{9}{8}\dfrac{(3-|a_0|)^2(1-|a_0|^2)^2}{((3-|a_0|)^2-1)((3-|a_0|)^2-|a_0|^4)}\\&\quad+\lambda\left(\dfrac{(3-|a_0|)^2(1-|a_0|^2)^2}{((3-|a_0|)^2-1)((3-|a_0|)^2-|a_0|^4)}\right)^2\\&=1-\dfrac{(1-|a_0|)^3(1+|a_0|)\eta(|a_0|)}{8(4-|a_0|)^2(2-|a_0|)^2(9-6|a_0|+|a_0|^2-|a_0|^4)^2},
	\end{align*}
	where
	\begin{align*}
		\eta(t)&:=(207-84t+41t^2+24t^3-8t^4)(2-t)(4-t)(9-6t+t^2-t^4)\\&\quad-8\lambda(3-t)^4(1+t)^3(1-t).
	\end{align*}
	It is easy to see that the function $\eta(t)$ has exactly one stationary point $\gamma=0.571317...$, which is the unique root of the equation $\eta^\prime(t)=0$ in $[0,1]$. It is worth pointing out that $\eta^\prime(t)=0$ is equivalent to 
	\begin{align*}
		&-13581+22491t-14085t^2-2584t^3+8565t^4-5130t^5+903t^6+636t^7\\&\quad-324t^8+40t^9-8\lambda(3-t)^3(1+t)^2(1-7t+4t^2)=0
	\end{align*}
	from which the value of $\lambda$ is obtained which mentioned in the statement of the theorem.\vspace{1.2mm}

We must demonstrate $\eta(\gamma)=0$. The equation $\eta^\prime(\gamma)=0$ is fulfilled (in fact, $\lambda$ was chosen in such the way at the beginning).
	We put the value of $\lambda$ in the expression of $\eta(t)$ which yields that 
	\begin{align*}
	\eta(t)=\dfrac{Q(t)}{(1-7t+4t^2)},
	\end{align*}
	where \begin{align*}
	Q(t)&:=55647-212544t+296244t^2-200754t^3+61377t^4+5198t^5-10420t^6\\&\quad+972t^7+960t^8+408t^9-420t^{10}+100t^{11}-8t^{12}.
	\end{align*}
	
	As $\gamma=0.571317...$ is the unique root of the equation $Q(t)=0$ in the interval $[0,1]$, so $\eta(\gamma)=0$. Besides this observation, we see that $\eta(0)>0$ and $\eta(1)>0$. Thus $\eta(t)\geq 0$ in the interval $[0,1]$ which proves that $\mathcal{A}_{3,f}(r)\leq1$ for $0\leq|a_0|\leq1$ and $r\leq1/3$.\vspace{1.2mm}
	To prove the constant $\lambda$ is sharp, we consider the function $ f_a $ given by \eqref{e-5.1}. For this $f_a$ and $r=1/3$, we see that
	\begin{align*}
		\mathcal{A}_{3,f_a}(1/(3-a))&=1+J_{ \lambda}(a)+(\lambda_1-\lambda)\left(\dfrac{(3-a)^2(1-a^2)^2}{(4-a)(2-a)(9-6a+a^2-a^4)}\right)^2,
	\end{align*}
	where \begin{align*}
		J_{ \lambda}(a)&:=-1+a^2+\dfrac{1-a^2}{2-a}+\dfrac{9}{8}\dfrac{(3-a)^2(1-a^2)^2}{(4-a)(2-a)(9-6a+a^2-a^4)}\\&\quad+\lambda\left(\dfrac{(3-a)^2(1-a^2)^2}{(4-a)(2-a)(9-6a+a^2-a^4)}\right)^2.
		\end{align*}
	Choose $a=\gamma$, where $ \gamma $ is the unique root of the equation $Q(t)=0$ in $(0,1)$. Clearly, we have $J_{\lambda}(\gamma)=0$, and hence $ \mathcal{A}_{3,f_a}(1/(3-a))=1+l(\lambda_1-\lambda), $
	where $l>0$. Therefore, $\mathcal{A}_{3,f_a}(1/(3-a))>1$ in case  $\lambda_1>\lambda$. This proves the sharpness assertion. This concludes the proof.
\end{proof}
\section{Bohr-type inequalities for subordination and K-quasiconformal harmonic mappings}
For a continuously differentiable complex-valued mapping $ f(z)=u(x,y)+iv(x,y) $, $ z=x+iy $, we use the common notions for its formal derivatives:
\begin{align*}
	f_{z}=\frac{1}{2}\left(f_x-if_y\right)\;\; \mbox{and}\;\; f_{\bar{z}}=\frac{1}{2}\left(f_x+if_y\right),
\end{align*}
where $ u $ and $ v $ are real-valued harmonic functions on $ \mathbb{D} $. We say that $ f $ is a harmonic mapping in a simply connected domain $ \Omega $ if $ f $ is twice continuously differentiable and satisfies the Laplacian equation $ \Delta f=4f_{z\bar{z}}=0 $ in $ \Omega $, where $ \Delta $ is the complex Laplacian operator defined by
$ \Delta:={\partial^2}/{\partial x^2}+{\partial^2}/{\partial y^2}. $ It follows that $ f $ admits the canonical representation $ f=h+\bar{g} $, where $ h $ and $ g $ are analytic in $ \mathbb{D} $ with $ f(0)=h(0) $. The Jacobian $ J_f $ of $ f $ is defined by $ J_f:=|h^{\prime}|^2-|g^{\prime}|^2 $. We say that $ f $ is sense-preserving in $ \mathbb{D} $ if $ J_f(z)>0 $ in $ \mathbb{D} $. As a matter of fact, $ f $ is locally univalent and sense-preserving if and only if $ J_f(z)>0 $ in $ \mathbb{D} $, or equivalently if $ h^{\prime}\neq 0 $ in $ \mathbb{D} $ and the dilation $ \omega_f:=g^{\prime}/h^{\prime} $ has the property that $ |\omega_f(z)|<1 $ (see \cite{Lew-BAMS-1936}). In view of this definition, we see that the class of analytic mappings is a subclass of the class of harmonic mappings. Accordingly, it can be concluded that harmonic mappings have a set of properties that do not hold for analytic functions. As compared to the classes of analytic functions, a little is known about the Bohr's phenomenon for the classes of harmonic mappings. Therefore, it is imperative to examine Bohr's phenomenon specifically for certain classes of harmonic mappings.\vspace{1.2mm}

\par It is well-known that harmonic mappings are famous for their use in the study of minimal surfaces and also play important roles in a variety of problems in applied mathematics (e.g., see Choquet\cite{Choquet-1945}, Dorff \cite{Dorff-2003}, Duren \cite{Duren-Harmonic-2004} or Lewy \cite{Lew-BAMS-1936}). More precisely, harmonic mappings play the natural role in parameterizing minimal surfaces in the context of differential geometry. However, planar harmonic mappings have applications not only in the differential geometry but also in various fields of engineering, physics, operations research and other intriguing aspects of applied mathematics. The theory of harmonic mappings has been used to study and solve fluid flow problems (see \cite{Aleman-2012}). The theory of univalent harmonic mappings having prominent geometric properties like starlikeness, convexity and close-to-convexity appears naturally while dealing with planner fluid dynamical problems. For instance, the fluid flow problem on a convex domain satisfying an interesting geometric property has been extensively studied by Aleman and Constantin  \cite{Aleman-2012}. With the help of geometric properties of harmonic mappings, Constantin and Martin \cite{constantin-2017} have obtained a complete solution of classifying all two dimensional fluid flows.\vspace{1.2mm}

A sense-preserving homeomorphism $f$ from the unit disk $\mathbb{D}$ onto $\Omega^{\prime}$, contained in the Sobolev class $W^{1,2}_{loc}(\mathbb{D})$, is said to be a $K$-quasiconformal mapping if, for $z\in\mathbb{D},$
\begin{align*}
\dfrac{|f_z|+|f_{\overline{z}}|}{|f_z|-|f_{\overline{z}}|}=\dfrac{1+|\omega_f(z)|}{1-|\omega_f(z)|}\leq K,\;\mbox{i.e.},\; |\omega_f(z)|\leq k=\dfrac{K-1}{K+1},
\end{align*} 
where $K\geq 1$ so that $k\in [0,1)$ (see\cite{Lehto-Virtanen-1973,Vuorinen-1988}). In mathematical complex analysis, a quasiconformal mapping, introduced by Gr\"otzsch in $ 1928 $ \cite{Grotzsch-1928} and Ahlfors \cite{Ahlfors-1935-AM} in $ 1935 $, is a homeomorphism between plane domains which to first order takes small circles to small ellipses of bounded eccentricity. Intuitively, let $ f : D\rightarrow \Omega^{\prime} $ (where $ D $ and $ \Omega^{\prime} $ are two domains in $ \mathbb{C} $) be an orientation-preserving homeomorphism between open sets in the plane. If $ f $ is continuously differentiable, then it is $ K $-quasiconformal if the derivative of $f$ at every point maps circles to ellipses with eccentricity bounded by $ K $.  Ahlfors \cite{Ahlfors-1935-AM} and Bers \cite{Bers-BAMS-1997} established the measurable Riemann mapping theorem, which is of great importance in the study of quasiconformal mappings in two dimensions. The theorem generalizes the Riemann mapping theorem from conformal to quasiconformal homeomorphisms. In the recent years, quasiconformal geometry has attracted attention from different fields, such as applied mathematics, computer vision and medical imaging. Computational quasiconformal geometry has been developed, which extends the quasiconformal theory into a discrete setting. \vspace{1.2mm} 

For the class of $K$-quasiconformal sense-preserving harmonic mappings in $\mathbb{D},$ Kayumov \emph{et al.} \cite{kayumov-Ponnusamy-Shakirov-2017} have established the Bohr's inequality finding the sharp Bohr radius.
\begin{thm}\label{th-5.1}\cite{kayumov-Ponnusamy-Shakirov-2017}
	Suppose that $f(z)=h(z)+\overline{g(z)}=\sum_{n=0}^{\infty}a_nz^n+\overline{\sum_{n=1}^{\infty}b_nz^n}$ is a $K$-quasiconformal sense-preserving harmonic mapping in $\mathbb{D},$  where $h$ is bounded analytic function in $\mathbb{D}.$  Then
	\begin{align*}
		|a_0|+\sum_{n=1}^{\infty}(|a_n|+|b_n|)r^n\leq {||h||}_{\infty}\;\;\mbox{for}\;\; r\leq\dfrac{K+1}{5K+1}.
	\end{align*} 
The constant $(K+1)/(5K+1)$ is sharp. Moreover, 
\begin{align*}
	|a_0|^2+\sum_{n=1}^{\infty}(|a_n|+|b_n|)r^n\leq {||h||}_{\infty}\;\;\mbox{for}\;\; r\leq\dfrac{K+1}{3K+1}.
\end{align*} 
The constant $(K+1)/(3K+1)$ is sharp
\end{thm}

Moreover, the authors  have remarked in \cite{kayumov-Ponnusamy-Shakirov-2017} that the boundedness condition on $ h $ in Theorem \ref{th-5.1} can be replaced by half-plane condition and also, the Bohr radius remains the same in this case too.
\begin{thm}\label{th-5.2}\cite{kayumov-Ponnusamy-Shakirov-2017}
	Suppose that $f(z)=h(z)+\overline{g(z)}=\sum_{n=0}^{\infty}a_nz^n+\overline{\sum_{n=1}^{\infty}b_nz^n}$ is a $K$-quasiconformal sense-preserving harmonic mapping in $\mathbb{D},$  where $h$ is satisfies the condition $Re (h(z))\leq 1$ in $\mathbb{D}.$ Then
	\begin{align*}
		|a_0|+\sum_{n=1}^{\infty}(|a_n|r^n+|b_n|)r^n\leq 1\;\;\mbox{for}\;\; r\leq\dfrac{K+1}{5K+1}.
	\end{align*} 
	The constant $(K+1)/(5K+1)$ is sharp.
\end{thm}
In this section, we obtain the next two results showing that improved Bohr radius for $ K $-quasiconformal sense-preserving harmonic mappings in the unit disk $ \mathbb{D} $ can be obtained. We use some idea of the proofs of \cite[Theorems 1.1 and 1.2 ]{kayumov-Ponnusamy-Shakirov-2017} and \cite[Theorem 1]{Ponnusamy-2017} to give the proof of our results. 
\begin{thm}\label{thh-5.3}
	Suppose that $f(z)=h(z)+\overline{g(z)}=\sum_{n=0}^{\infty}a_nz^n+\overline{\sum_{n=1}^{\infty}b_nz^n}$ is a $K$-quasiconformal sense-preserving harmonic mapping in $\mathbb{D},$  where $h$ is bounded analytic function in $\mathbb{D}$ and $k=(K-1)/(K+1),$ $K\geq 1$ so that $k\in [0,1).$ Then
\begin{align}\label{ee-5.11}
		\mathcal{I}^1_f(a,r,k):=|h(z)|+\sum_{n=1}^{\infty}(|a_n|+|b_n|)r^n\leq {||h||}_{\infty}\;\;\mbox{for}\;\; |z|=r\leq r_1(k),
\end{align} 
where $r_1(k)$ is the unique root of the equation $2(k+1)r(1+r)-(1-r)^2=0$ in the interval $ (0,1) $. The radius $r_1(k)$ is sharp. Moreover, 
\begin{align}\label{ee-33.22}
	\mathcal{I}^2_f(a,r,k):=|h(z)|^2+\sum_{n=1}^{\infty}(|a_n|+|b_n|)r^n\leq {||h||}_{\infty}\;\;\mbox{for}\;\; |z|=r\leq r_2(k),
\end{align}
where $r_2(k)$ is the unique root of the equation $(k+1)r(1+r)-(1-r)^2=0$ in the interval $ (0,1) $ and ${||h||}_{\infty}:=\sup_{z\in\mathbb{D}}|h(z)|$.  The radius $r_2(k)$ is sharp. 
\end{thm}
\begin{thm}\label{thh-5.4}
	Suppose that $f(z)=h(z)+\overline{g(z)}=\sum_{n=0}^{\infty}a_nz^n+\overline{\sum_{n=1}^{\infty}b_nz^n}$ is a $K$-quasiconformal sense-preserving harmonic mapping in $\mathbb{D},$  where $h$ is satisfies the condition $Re( h(z))\leq 1$ in $\mathbb{D}$ with $h(0)>0$ and $k=(K-1)/(K+1),$ $K\geq 1$ so that $k\in [0,1).$ Then
	\begin{align*}
		\mathcal{J}_f(a,r,k):=|h(z)|+\sum_{n=1}^{\infty}\left(|a_n|+|b_n|\right)r^n\leq 1\;\;\mbox{for}\;\; r\leq\dfrac{K+1}{7K+3}.
	\end{align*} 
	The constant $(K+1)/(7K+3)$ is sharp. 
\end{thm}
The following lemma will play a key role in the proof of Theorems \ref{thh-5.3} and \ref{thh-5.4}.
\begin{lem}\label{lem-5.1}\cite{kayumov-Ponnusamy-Shakirov-2017}
	Suppose that $h(z)=\sum_{n=0}^{\infty}a_nz^n$ and $g(z)=\sum_{n=0}^{\infty}b_nz^n$ are analytic functions in the unit disk $\mathbb{D}$ such that $|g^{\prime}(z)|\leq k|h^{\prime}(z)|$ in $\mathbb{D}$ and for some $k\in[0,1].$ Then
	\begin{align*}
		\sum_{n=1}^{\infty}|b_n|^2r^n\leq k^2 \sum_{n=1}^{\infty}|a_n|^2r^n\;\;\mbox{for}\;\;|z|=r<1.
	\end{align*}
\end{lem}
\begin{proof}[\bf Proof of Theorem \ref{thh-5.3}]
	By using the Schwarz-Pick lemma for the function $h,$ we have 
	\begin{align*}
		|h(z)|\leq \dfrac{r+|a_0|}{1+r|a_0|}\;\;\mbox{for}\;\;|z|=r.
	\end{align*}
For simplicity, we suppose that ${||h||}_{\infty}=1.$ Then we must have  $|a_n|\leq 1-|a_0|^2$ for $n\geq 1.$ Let $\omega_f$ denote the dilatation of $f=h+\overline{g}$ so that $|g^{\prime}(z)|\leq k|h^{\prime}(z)|$ in $\mathbb{D},$ where $k\in[0,1)$. Thus, based on the Lemma \ref{lem-5.1}, it follows that 
\begin{align*}
	\sum_{n=1}^{\infty}|b_n|^2r^n\leq k^2\sum_{n=1}^{\infty}|a_n|^2r^n\leq k^2(1-|a_0|^2)^2\dfrac{r}{1-r}.
\end{align*} 
Consequently, we see that
\begin{align*}
	\sum_{n=1}^{\infty}|b_n|r^n\leq \sqrt{\sum_{n=1}^{\infty}|b_n|^2r^n}\sqrt{\sum_{n=1}^{\infty}r^n}\leq k(1-|a_0|^2)\dfrac{r}{1-r}.
\end{align*}
Let $|a_0|=a\in[0,1).$
Thus, an easy computation gives us
\begin{align*}
	\mathcal{I}^1_f(a,r,k)&\leq \dfrac{r+a}{1+ra}+(1-a^2)(1+k)\dfrac{r}{1-r}\\&=1+(1-a)\left(\dfrac{r-1}{1+ra}+(1+a)(1+k)\dfrac{r}{1-r}\right)=1+\dfrac{(1-a)R_1(a,r,k)}{(1+ra)(1-r)},
\end{align*}
where $R_1(a,r,k):=-(r-1)^2+(1+a)(1+k)r(1+ra).$
We claim that $R_1(a,r,k)\leq 0$ for $r\leq r_1(k)$ and $a\in[0,1).$ Since $a\in[0,1)$, we see that 
\begin{align*}
	R_1(a,r,k)\leq -(r-1)^2+2(1+k)r(1+r)\leq 0\;\;\mbox{for}\;\; r\leq r_1(k),
\end{align*}
and therefore, $\mathcal{I}^1_f(a,r,k)\leq 1$ for $r\leq r_1(k)$, where $r_1(k)$ is the root of the equation  $2(1+k)r(1+r)-(r-1)^2=0$ in the interval $[0,1).$
To prove the sharpness, consider the function $f_0=h_0+\overline{g_0}$, and
\begin{align}\label{ee-5.1}
	h_0(z)=\dfrac{a-z}{1-az}=a+\sum_{n=1}^{\infty}a_nz^n,\;\;\mbox{where}\;\; a_n=-(1-a^2){a}^{n-1}\;\mbox{for}\; n\geq 1,
\end{align}
and $g_0(z)=\lambda kh_0(z),$  $|\lambda|=1,$ $a\in[0,1)$ and $k=(K-1)/(K+1).$ Then by a simple computation shows that
\begin{align*}
	|h_0(-r)|+\sum_{n=1}^{\infty}(|a_n|+|b_n|)r^n=\dfrac{a+r}{1+ra}+(1-a^2)(1+k)\sum_{n=1}^{\infty}a^{n-1}r^n=1+\dfrac{(1-a)U(a,r,k)}{(1-a^2r^2)},
\end{align*}
where $U(a,r,k):=(r-1)(1-ar)+(1+a)(1+k)r(1+ar).$ We see that $\lim_{a\rightarrow 1^-}U(a,r_1(k),k)=0$ and hence $\mathcal{I}^1_f(r)>1$ for $r>r_1(k).$\vspace{1.2mm}\\
Moreover, 
\begin{align*}
	\mathcal{I}^2_f(a,r,k)&\leq \left(\dfrac{r+a}{1+ra}\right)^2+(1-a^2)(1+k)\dfrac{r}{1-r}\\&=1+\dfrac{(1-a^2)R_2(a,r,k)}{(1+ar)^2(1-r)}
\end{align*}
where $R_2(a,r,k):=(r^2-1)(1-r)+(1+k)r(1+ra)^2.$  By the similar arguments use as above, we can easily seen that	$R_2(a,r,k)\leq0$ for $r\leq r_2(k)$, where $r_2(k)$ is the unique root in $(0,1)$ of the equation $(1+k)r(1+r)-(1-r)^2=0.$ The sharpness of $r_2(k)$, we can easily shown for the function $f_0=h_0+\overline{g}_0$ given by  \eqref{ee-5.1} and we omit the details.
\end{proof}
For sense-preserving harmonic mapping $f(z)=h(z)+\overline{g(z)}$ in $\mathbb{D}$, the dilatation $\omega_f=g^{\prime}/h^{\prime}$ has the property that  $|\omega_f|<1$ in $\mathbb{D}$ \cite{Lehto-Virtanen-1973}.
Therefore, as a consequence of Theorem \ref{thh-5.3}, we obtain the following corollary.
\begin{cor}
	Suppose that $f(z)=h(z)+\overline{g(z)}=\sum_{n=0}^{\infty}a_nz^n+\overline{\sum_{n=1}^{\infty}b_nz^n}$ is a sense-preserving harmonic mapping in $\mathbb{D},$  where $h$ is bounded analytic function in $\mathbb{D}.$ Then
	the inequality \eqref{ee-5.11} holds for $r\leq r_1$  where $r_1$ is the unique root of the equation $3r^2+6r-1=0$ in the interval $(0,1).$ Also the inequality \eqref{ee-33.22} holds for $r\leq r_2$, where $r_2$ is the unique root of the equation $r^2+4r-1=0$ in $(0,1).$ 
\end{cor}
\begin{proof}
	Allow $k\rightarrow1$ in Theorem \ref{thh-5.3} to prove this corollary.
\end{proof}
\begin{proof}[\bf Proof of Theorem \ref{thh-5.4}]
	We recall that if $p(z)=\sum_{n=0}^{\infty}p_nz^n$ is analytic in $\mathbb{D}$ such that $Re (p(z))>0$ in $\mathbb{D}$ , then $|p_n|\leq 2Re (p_0)$ for all $n\geq 1.$ Applying this result to $p(z)=1-f(z)$ leads to $|a_n|\leq 2(1-a_0) $ for all $n\geq 1.$ By the Lemma \ref{lem-5.1}, we obtain that 
	\[
	\begin{cases}
	\displaystyle\sum_{n=1}^{\infty}|b_n|^2r^n\leq k^2\sum_{n=1}^{\infty}|a_n|^2r^n\leq 4k^2(1-a_0)^2\dfrac{r}{1-r},\vspace{2mm}\\
	\displaystyle\sum_{n=1}^{\infty}|b_n|r^n\leq \sqrt{\sum_{n=1}^{\infty}|b_n|^2r^n}\sqrt{\sum_{n=1}^{\infty}r^n}\leq 2k(1-a_0)\dfrac{r}{1-r}.
	\end{cases}
	\]
Moreover, it is easy to see that
\begin{align*}
|h(z)|\leq |h(0)|+|h(z)-h(0)|\leq a_0+\sum_{n=1}^{\infty}|a_n|r^n\leq a_0+2(1-a_0)\dfrac{r}{1-r}.
\end{align*}
In view of the above estimates, we obtain
\begin{align*}
	\mathcal{J}_f(a,r,k)&\leq a_0+4(1-a_0)\dfrac{r}{1-r}+2k(1-a_0)\dfrac{r}{1-r}\\&=1+(1-a_0)\left(-1+4\dfrac{r}{1-r}+2k\dfrac{r}{1-r}\right)
\end{align*}which is less than or equal to $1$ for $r\leq1/(5+2k).$ Again, substituting $k=(K-1)/(K+1)$ gives the desired result. Moreover, sharpness can be easily seen by the functions of the form
\begin{align*}
	h(z)=\dfrac{a-z}{1-az},\; 0<a<1,\;\mbox{and}\; g(z)=k\lambda h(z),\; |\lambda|=1,
\end{align*}
and we omit the details.
\end{proof}
We obtain the following corollary of Theorem \ref{thh-5.4} for a sense-preserving harmonic mapping in $\mathbb{D}$ due to $k\rightarrow 1.$
\begin{cor}
Suppose that $f(z)=h(z)+\overline{g(z)}=\sum_{n=0}^{\infty}a_nz^n+\overline{\sum_{n=1}^{\infty}b_nz^n}$ is a sense-preserving harmonic mapping in $\mathbb{D},$  where $h$ is satisfies the condition $Re(h(z)) \leq 1$ in $\mathbb{D}$ and $f(0)\in(0,1).$ Then
\begin{align*}
|h(z)|+\sum_{n=1}^{\infty}(|a_n|+|b_n|)r^n\leq 1\;\;\mbox{for}\;\; r\leq \dfrac{1}{7}
\end{align*}
and the number $1/7$ is sharp.
\end{cor}
\section{Bohr-Rogosinski-type inequalities for classes of subordination}
It is worth mentioning that various questions (including the refined forms) on a related new concept called the Bohr-Rogosinski's phenomenon are currently being studied (see e.g., \cite{Huang-Liu-Ponnu-AMP-2020,Liu-Liu-Ponnusamy-2021}). However, to the best of our knowledge, unlike the refined versions of the Bohr's inequality for different classes of functions, there is no strengthened version of the Rogosinski's inequality which holds for $ r\leq 1/2 $ and for all $ N\in\mathbb{N} $. We may generalize the Bohr-Rogosinski radius for the class $\mathcal{B}$, by writing the Bohr-Rogosinski inequality is in the following equivalent form
\begin{align*}
	\sum_{n=1}^{\infty}|b_n|r^n\leq 1-|g(z)|=dist(g(z),\partial\mathbb{D}),
\end{align*}
where the number $1-|g(z)|$ is the distance from the point $g(z)$ to the boundary $\partial \mathbb{D}$ of the unit disk $\mathbb{D}.$ Using this distance formulation of the Bohr-Rogosinski inequality, the notion of Bohr-Rogosinski radius can be generalized to the class of functions $f$ analytic in $\mathbb{D}$ which take values in a given domain $D$. For our formulation, we shall use the notion of subordination.\vspace{1.2mm}

As in case of Bohr's phenomenon (see \cite{Abu-CVEE-2010}), for  a given $f,$  it is natural to introduce the subordination family: $S(f)=\{g:g\prec f\}$, where $\prec$ denotes the usual subordination relation (see \cite{Duren-1983,Goluzin-Trudy-1951}). 
\begin{defn}
	We say that the family $S(f)$ has a Bohr-Rogosinski's phenomenon if there exists  $r_f$, $0<r_f\leq1,$ such that whenever $g(z)=\sum_{n=0}^{\infty}b_nz^n\in S(f),$ we have 
	\begin{align}\label{e-4.1}
		|g(z)|+\sum_{n=1}^{\infty}|b_n|r^n\leq |f(0)|+dist(f(0), \partial f(\mathbb{D}))\;\;\mbox{for}\; |z|=r\leq r_f,
	\end{align}
	where $dist(f(0), \partial f(\mathbb{D}))$ is the Euclidean distance between $f(0)$ and the boundary of $ f(\mathbb{D}).$ 
	We observe that if $f(z)=(a_0-z)/(1-\overline{a_0}z)$ with $|a_0|<1$, and $\Omega=\mathbb{D},$ we have $dist(f(0), \partial f(\mathbb{D}))=1-|f(0)|,$ which means that \eqref{e-4.1} holds with $r_f=\sqrt{5}-2$, according to \cite[Theorem 1]{Abu-CVEE-2010}.
\end{defn}
In view of this distance, the Bohr-Rogosinski's theorem has been extended in \cite{ Ponnusamy-2017} to a variety of distances provided the Bohr-Rogosinski's phenomenon exists and the following two results are obtained.
\begin{thm}\cite{Ponnusamy-2017}\label{th-4.1}
	If $f$, $g$ are analytic in $\mathbb{D}$ such that $f$ is univalent in $\mathbb{D}$ and $g(z)=\sum_{n=0}^{\infty}b_nz^n\in S(f),$ then inequality \eqref{e-4.1} holds with $r_f=5-2\sqrt{6}$. The sharpness of $r_f$ is shown by the Koebe function $f(z)=z/(1-z)^2.$
\end{thm}
\begin{thm}\cite{Ponnusamy-2017}\label{th-4.2}
	If $f$, $g$ are analytic in $\mathbb{D}$ such that $f$ is univalent and convex in $\mathbb{D}$ and $g(z)=\sum_{n=0}^{\infty}b_nz^n\in S(f),$ then inequality \eqref{e-4.1} holds with $r_f=1/5$. The sharpness of $r_f$ is shown by the function $f(z)=z/(1-z).$
\end{thm}
Continuing the study, Liu and Ponnusamy \cite{Liu-Ponnusamy-BMMS-2019} investigated for the sharp Bohr's inequality for $K$-quasiconformal  harmonic mappings and also for the subordination class of function whose real part is subordinate to a univalent and convex functions. The authors have established the following results. 
\begin{thm}\cite{Liu-Ponnusamy-BMMS-2019}
	Suppose that $f(z)=h(z)+\overline{g(z)}=\sum_{n=0}^{\infty}a_nz^n+\overline{\sum_{n=1}^{\infty}b_nz^n}$ is a $K$-quasiconformal sense-preserving harmonic mapping in $\mathbb{D}$ and $h\prec\psi,$ where $\psi$ is univalent and convex in $\mathbb{D}.$ Then
	\begin{align*}
		\sum_{n=1}^{\infty}(|a_n|+|b_n|)r^n\leq dist(\psi(0),\partial\psi(\mathbb{D}))\;\;\mbox{for}\;\; |z|=r\leq\dfrac{K+1}{5K+1}.
	\end{align*} 
	The result is sharp.
\end{thm}
\begin{thm}\cite{Liu-Ponnusamy-BMMS-2019}
Suppose that $f(z)=h(z)+\overline{g(z)}=\sum_{n=0}^{\infty}a_nz^n+\overline{\sum_{n=1}^{\infty}b_nz^n}$ is a $K$-quasiconformal sense-preserving harmonic mapping in $\mathbb{D}$ and $h\prec\psi,$ where $\psi$ is analytic and univalent in $\mathbb{D}.$ Then
\begin{align*}
\sum_{n=1}^{\infty}(|a_n|+|b_n|)r^n\leq dist(\psi(0),\partial\psi(\mathbb{D}))
\end{align*} 
for $|z|=r\leq r_u(k),$ where $r_u(k)$ is the root of the equation $(1-r)^2-4r(1+k\sqrt{1+r})=0$ in the interval $(0,1)$ and $k=(K-1)/(K+1).$
\end{thm}
Following that, we will look at the Bohr-Rogosinski's phenomenon for the family $S(\psi)$ . We say that the family $S(\psi)$ has Bohr-Rogosinski's phenomenon if there exists $r_{\psi},$ $0<r_{\psi}\leq 1,$ such that whenever $f(z)=h(z)+\overline{g(z)}=\sum_{n=0}^{\infty}a_nz^n+\overline{\sum_{n=1}^{\infty}b_nz^n}$ and  $h(z)=\sum_{n=0}^{\infty}a_nz^n\prec\psi(z),$ we have  
\begin{align}\label{ee-5.2}
|h(z)|+\sum_{n=1}^{\infty}(|a_n|+|b_n|)r^n\leq |\psi(0)|+ dist(\psi(0),\partial\psi(\mathbb{D})).
\end{align}
\par We obtain two results in this section for $K$-quasiconformal sense-preserving harmonic mapping in $\mathbb{D}$ and $h\prec \psi,$ where $\psi$ is univalent and convex in $\mathbb{D}.$ We state the first result here.
\begin{thm}\label{th-5.3}
Suppose that $f(z)=h(z)+\overline{g(z)}=\sum_{n=0}^{\infty}a_nz^n+\overline{\sum_{n=1}^{\infty}b_nz^n}$ is a $K$-quasiconformal sense-preserving harmonic mapping in $\mathbb{D}$ and $h\prec \psi,$ where $\psi$ is univalent and convex in $\mathbb{D}.$ Then the inequality \eqref{ee-5.2} holds for  $|z|=r\leq r_{u,c}=(K+1)/(7K+3)$.
The result is sharp. 
\end{thm}
We obtain the following corollary which is an immediate consequence of Theorem \ref{th-5.3}. More precisely, we obtain a corollary showing the sharp Bohr-Rogosinski inequality for a sense-preserving harmonic mapping $f=h+\overline{g}$ in $\mathbb{D}$ and $h\prec \psi,$ where $\psi$ is univalent and convex in $\mathbb{D}.$
\begin{cor}\label{cor-5.3}
	Suppose that $f(z)=h(z)+\overline{g(z)}=\sum_{n=0}^{\infty}a_nz^n+\overline{\sum_{n=1}^{\infty}b_nz^n}$ is a sense-preserving harmonic mapping in $\mathbb{D}$ and $h\prec \psi,$ where $\psi$ is univalent and convex in $\mathbb{D}.$ Then the inequality
	\eqref{ee-5.2} holds for $|z|=r\leq1/7.$
	The number $1/7$ is sharp. 
\end{cor}
\begin{proof}[\bf Proof of Corollary \ref{cor-5.3}]
	Allow $k=1$ in the proof of Theorem \ref{th-5.3}. Indeed, since $ f(z)$ is locally univalent and sense-preserving in 
	$\mathbb{D},$ we have $|g^{\prime}(z)|<|h^{\prime}(z)|$ in $\mathbb{D}$ and thus, we can allow $K\rightarrow\infty$ to obtain the desired conclusion.  
\end{proof}
The next theorem is the second result of this section.
\begin{thm}\label{th-5.4}
Suppose that $f(z)=h(z)+\overline{g(z)}=\sum_{n=0}^{\infty}a_nz^n+\overline{\sum_{n=1}^{\infty}b_nz^n}$ is a $K$-quasiconformal sense-preserving harmonic mapping in $\mathbb{D}$ and $h\prec \psi,$ where $\psi$ is univalent in $\mathbb{D}.$ Then the inequality
\eqref{ee-5.2} holds for $|z|=r\leq r_u(k)$, where $r_u(k)$ is the unique root of the equation 
$8r+4kr\sqrt{1+r}-(1-r)^2=0$ in $(0,1).$
\end{thm}
We obtain the next corollary of Theorem \ref{th-5.4}.
\begin{cor}\label{cor-5.5}
	Suppose that $f(z)=h(z)+\overline{g(z)}=\sum_{n=0}^{\infty}a_nz^n+\overline{\sum_{n=1}^{\infty}b_nz^n}$ is a sense-preserving harmonic mapping in $\mathbb{D}$ and $h\prec \psi,$ where $\psi$ is analytic and univalent in $\mathbb{D}.$ Then the inequality \eqref{ee-5.2} holds for $|z|=r\leq r^*_u,$ where $r^*_u$ is the root of the equation $8r+4r\sqrt{1+r}-(1-r)^2=0$ in the interval $(0,1).$
\end{cor}
\begin{proof}[\bf Proof of Corollary \ref{cor-5.5}]
	To prove Corollary \ref{cor-5.5}, allow $k\rightarrow 1$ in Theorem \ref{th-5.4}.
\end{proof}
We now discuss the proof of Theorems \ref{th-5.3} and \ref{th-5.4} in detail.
\begin{proof}[\bf Proof of Theorem \ref{th-5.3}]
	 Let $h(z)=\sum_{n=0}^{\infty}a_nz^n\prec \psi(z) ,$ where $\psi$ is a univalent and convex mapping in $\mathbb{D}$ onto a convex domain $\psi(\mathbb{D}).$	Then it is well known  that 
 (see \cite{Duren-1983,Goluzin-Trudy-1951}), for all $z\in\mathbb{D}$ and $n\geq 1$,  
\begin{align*}
	\dfrac{1}{2}|\psi^{\prime}(z)|(1-|z|^2)\leq dist(\psi(0),\partial\psi(\mathbb{D}))\leq |\psi^{\prime}(z)|(1-|z|^2),\; \mbox{and}\; |a_n|\leq |\psi^{\prime}(z)|.
\end{align*} 
Plugging $z=0$ in above inequalities and thus we obtain $|a_n|\leq 2dist(\psi(0),\partial\psi(\mathbb{D}))$ for $n\geq1.$ 
Consequently, 
\begin{align*}
	\sum_{n=1}^{\infty}|a_n|r^n\leq 2\;dist(\psi(0),\partial\psi(\mathbb{D}))\sum_{n=1}^{\infty}r^n. 
\end{align*}
Because $f=h+\overline{g}$ is a $K$-quasiconformal sense-preserving harmonic mapping so that $|g^{\prime}(z)|\leq k|h^{\prime}(z)|$ in $\mathbb{D}$, where $0\leq k\leq 1$, by Lemma \ref{lem-5.1} and Cauchy-Schwarz inequality, it follows that 
\begin{align*}
	\sum_{n=1}^{\infty}|b_n|r^n&\leq\sqrt{\sum_{n=1}^{\infty}|b_n|^2r^n}\sqrt{\sum_{n=1}^{\infty}r^n}\leq k\sqrt{\sum_{n=1}^{\infty}|a_n|^2r^n}\sqrt{\sum_{n=1}^{\infty}r^n}\\&\leq 2k\;dist(\psi(0),\partial\psi(\mathbb{D}))\sum_{n=1}^{\infty}r^n.
\end{align*}
Also, because $h\prec \psi, $ it follows that $h(0)=\psi(0)$ and 
\begin{align*}
	|h(z)|\leq |h(0)|+|h(z)-h(0)|=|\psi(0)|+\bigg|\sum_{n=1}^{\infty}a_nz^n\bigg|\leq |\psi(0)|+2\;dist(\psi(0),\partial\psi(\mathbb{D}))\sum_{n=1}^{\infty}r^n.
\end{align*}
Thus, we obtain that 
\begin{align*}
|h(z)|+\sum_{n=1}^{\infty}(|a_n|+|b_n|)r^n&\leq |\psi(0)|+(4+2k)dist(\psi(0),\partial\psi(\mathbb{D}))\sum_{n=1}^{\infty}r^n\\&=|\psi(0)|+dist(\psi(0),\partial\psi(\mathbb{D}))\dfrac{(4+2k)r}{1-r}\\&\leq |\psi(0)|+dist(\psi(0),\partial\psi(\mathbb{D}))
\end{align*}
for $r\leq r_{u,c}= {1}/{(5+2k)}.$ Plugging $k={(K-1)}/{(K+1)}$ into $ r_{u,c} $ gives the desired result. To show the sharpness part, we consider the function $ 	\psi(z)=h(z)=\frac{1}{1-z}=\sum_{n=0}^{\infty}z^n $ and $g^{\prime}(z)=k\lambda h^{\prime}(z),$ where $\lambda\in \mathbb{D}.$ Then it is easy to see that 
\begin{align*}
	dist(\psi(0),\partial\psi(\mathbb{D}))=\dfrac{1}{2}, \;\; |\psi(0)|=1\;\;\mbox{and}\;\; g(z)=k\lambda\dfrac{z}{1-z}=k\lambda\sum_{n=1}^{\infty}z^n.
\end{align*}
A simple computation shows that 
\begin{align*}
	|h(r)|+\sum_{n=1}^{\infty}(|a_n|+|b_n|)r^n=\dfrac{1}{1-r}+\sum_{n=1}^{\infty}(1+k|\lambda|)r^n=\dfrac{1}{1-r}+(1+k|\lambda|)\dfrac{r}{1-r}
\end{align*}
which is bigger than or equal to $1+1/2$ if and only if 
\begin{align*}
	r\geq\dfrac{1}{5+2k|\lambda|}=\dfrac{K+1}{5K+5+2|\lambda|(K-1)}.
\end{align*}
This shows that the number $(K+1)/(7K+3)$ cannot be improved since $|\lambda|$ could be chosen so close to $1$ from left. This completes the proof.
\end{proof}
\begin{proof}[\bf Proof of Theorem \ref{th-5.4}]
	Let $h(z)=\sum_{n=0}^{\infty}a_nz^n\prec \psi(z) ,$ where $\psi$ is a univalent in $\mathbb{D}$ onto a simply connected domain $\psi(\mathbb{D}).$	Then it is well known  that 
	(see \cite{Duren-1983,Goluzin-Trudy-1951}), for all $z\in\mathbb{D}$ and $n\geq 1$,  
\begin{align*}
	\dfrac{1}{4}|\psi^{\prime}(z)|(1-|z|^2)\leq dist(\psi(0),\partial\psi(\mathbb{D}))\leq |\psi^{\prime}(z)|(1-|z|^2),\; \mbox{and}\; |a_n|\leq n|\psi^{\prime}(z)|.
\end{align*} 
	Plugging $z=0$ in above inequalities and thus, we obtain \begin{align*}
		|a_n|\leq 4n\hspace{0.3mm}dist(\psi(0),\partial\psi(\mathbb{D}))
	\end{align*} 
for $n\geq1.$ 	Consequently, we obtain
\begin{align*}
\sum_{n=1}^{\infty}|a_n|r^n\leq  4\hspace{0.3mm}dist(\psi(0),\partial\psi(\mathbb{D}))\dfrac{r}{(1-r)^2}. 
\end{align*}
Moreover, because $f=h+\overline{g}$ is a $K$-quasiconformal sense-preserving harmonic mapping so that $|g^{\prime}(z)|\leq k|h^{\prime}(z)|$ in $\mathbb{D}$, where $0\leq k\leq 1$, by Lemma \ref{lem-5.1} and Cauchy-Schwarz inequality, it follows that 
\begin{align*}
\sum_{n=1}^{\infty}|b_n|r^n&\leq 4k\hspace{0.3mm}dist(\psi(0),\partial\psi(\mathbb{D}))\sqrt{\sum_{n=1}^{\infty}n^2r^n}\sqrt{\sum_{n=1}^{\infty}r^n}\\&= 4k\hspace{0.3mm}dist(\psi(0),\partial\psi(\mathbb{D}))\sqrt{\dfrac{r(1+r)}{(1-r)^3}}\sqrt{\dfrac{r}{1-r}}\\&\leq 4k\hspace{0.3mm}dist(\psi(0),\partial\psi(\mathbb{D}))\dfrac{r\sqrt{1+r}}{(1-r)^2}.
\end{align*}
Also, because $h\prec \psi, $ it follows that $h(0)=\psi(0)$ and 
\begin{align*}
|h(z)|&\leq |h(0)|+|h(z)-h(0)|=|\psi(0)|+|\sum_{n=1}^{\infty}a_nz^n|\\&\leq |\psi(0)|+4\hspace{0.3mm}dist(\psi(0),\partial\psi(\mathbb{D}))\dfrac{r}{(1-r)^2}.
\end{align*}
Thus, we obtain 
\begin{align*}
|h(z)|+\sum_{n=1}^{\infty}(|a_n|+|b_n|)r^n&\leq |\psi(0)|+\left(\dfrac{8r}{(1-r)^2}+\dfrac{4kr\sqrt{1+r}}{(1-r)^2}\right)dist(\psi(0),\partial\psi(\mathbb{D}))
\end{align*}  
which is less than or equal to $|\psi(0)|+dist(\psi(0),\partial\psi(\mathbb{D}))$ if and only if 
\begin{align*}
	\dfrac{8r}{(1-r)^2}+\dfrac{4kr\sqrt{1+r}}{(1-r)^2}\leq 1.
\end{align*}
This gives $|z|=r\leq r_u,$ where $r_u=r_u(k)$ is as in the statement.
\end{proof}

\noindent{\bf Acknowledgment:} The authors are grateful to the referee(s) for providing insightful suggestions that contributed significantly to the improvement of the paper. 
\vspace{1.2mm}

\noindent\textbf{Compliance of Ethical Standards:}\\

\noindent\textbf{Conflict of interest.} The authors declare that there is no conflict  of interest regarding the publication of this paper.\vspace{1.2mm}


\noindent\textbf{Funding.} This research did not receive any specific grant from funding agencies in the public, commercial, or
not-for-profit sectors. \vspace{1.5mm}

\noindent\textbf{Author's contribution.} Each of the authors has contributed equally to preparing the manuscript.


\begin{thebibliography}{99}
	
\bibitem{Abu-CVEE-2010} {\sc Y. Abu-Muhanna},  Bohr's phenomenon in subordination and bounded harmonic classes, {\it Complex Var. Elliptic Equ.} {\bf  55} (2010), 1071--1078.

\bibitem{Aha-Aha-MJM-2023} {\sc M. B. Ahamed} and {\sc S. Ahammed}, Bohr Inequalities for Certain Classes of Harmonic Mappings, \textit{Mediterr. J. Math.} \textbf{21}, 21 (2024). 

\bibitem{Ahamed-RMJM-2021} {\sc M. B. Ahamed} and {\sc V. Allu}, Bohr phenomenon for certain classes of harmonic mappings, \textit{Rocky Mountain J. Math.} (2021)(to appear)


\bibitem{Ahamed-AMP-2021} {\sc M. B. Ahamed, V. Allu}, and {\sc H. Halder}, Bohr radius for certain classes of close-to-convex harmonic mappings,\textit{ Anal. Math. Phys.} (2021) 11:111. https://doi.org/10.1007/s13324-021-00551-y

\bibitem{Ahamed-CVEE-2021} {\sc M. B. Ahamed, V. Allu}, and {\sc H. Halder}, Improved Bohr inequalities for certain class of harmonic univalent functions, \textit{Complex Var. Elliptic Equ.} (2021) https://doi.org/10.1080/17476933.2021.1988583

\bibitem{Ahamed-AASFM-2022} {\sc M. B. Ahamed, V. Allu}, and {\sc H. Halder}, The Bohr phenomenon for analytic functions on shifted disks, \textit{Ann. Acad. Sci. Fenn. Math.} \textbf{47}(2022), 103-120. https://doi.org/10.54330/afm.112561

\bibitem{Aleman-2012} {\sc A. Aleman} and {\sc A. Constantin}, Harmonic maps and ideal fluid flows, {\it Arch. Ration. Mech. Anal.} {\bf 204} (2012), 479--513.

\bibitem{Aizn-PAMS-2000} {\sc L. Aizenberg}, Multidimensional analogues of Bohr's theorem on power series, \textit{Proc. Amer. Math. Soc.} {\bf 128} (2000), 1147--1155. https://doi.org/10.1090/S0002-9939-99-05084-4

\bibitem{Aizenberg-AMP-2012} {\sc L. Aizenberg}, Remarks on the Bohr and Rogosinski phenomenon for power series, \textit{Anal. Math. Phys.} \textbf{2} (2012), 69-78. DOI https://doi.org/10.1007/s13324-012-0024-7. 	

\bibitem{Aizenberg-Aytuna-Djakov-2001} {\sc L. Aizenberg, A. Aytuna}  and {\sc P. Djakov}, Generalization of theorem on Bohr for bases in spaces of holomorphic functions of several complex variables, {\it J. Math. Anal. Appl.} {\bf  258} (2001), 429--447.


\bibitem{Aizeberg-PLMS-2001} {\sc L. Aizenberg} and {\sc N. Tarkhanov}, A Bohr phenomenon for elliptic equations, \textit{Proc. London Math. Soc.} \textbf{82} (2)(2001), 385-401.  https://doi.org/10.1112/S0024611501012813.

\bibitem{Ahlfors-1935-AM} {\sc L. Ahlfors,} Zur Theorie der \"uberlagerungsfl\"achen, \textit{Acta Mathematica} (in German), \textbf{65} (1)(1935), 157-194, doi:10.1007/BF02420945

\bibitem{Ali & Abdul & Ng & CVEE & 2016} {\sc R. M. Ali}, {\sc Z. Abdulhadi} and {\sc Z. C. Ng}, The Bohr radius for starlike logharmonic mappings, \textit{Complex Var. Elliptic Equ.} \textbf{61}(1)(2016), 1--14.

\bibitem{Ali-2017} {\sc R. M. Ali, R.W. Barnard} and {\sc  A.Yu. Solynin}, A note on Bohr's phenomenon for power series, {\it J. Math. Anal. Appl.} {\bf 449} (2017), 154-167.

\bibitem{Allu-JMAA-2021} {\sc V. Allu} and {\sc H. Halder},  Bohr radius for certain classes of starlike and convex univalent functions, \textit{J. Math. Anal. Appl.} \textbf{493} (2021), 124519.https://doi.org/10.1016/j.jmaa.2020.124519

\bibitem{Allu-BSM-2021} {\sc V. Allu} and {\sc H. Halder}, Bohr phenomenon for certain subclasses of harmonic mappings, \textit{Bull. Sci. Math.} \textbf{173} (2021), 103053.https://doi.org/10.1016/j.bulsci.2021.103053

\bibitem{Allu-Halder-Banach} {\sc V. Allu} and {\sc H. Halder}, Bohr radius for Banach spaces on simply connected domains, arXiv preprint arXiv:2111.10880, 2021.

\bibitem{Allu-IM-2021} {\sc V. Allu} and {\sc H. Halder}, The Bohr inequality for certain harmonic mappings, Indag. Math. (2021) https://doi.org/10.1016/j.indag.2021.12.004.

\bibitem{Allu-CMB-2022} {\sc V. Allu} and {\sc H. Halder}, Operator valued analogue of multidimensional Bohr inequality, \textit{Canad. Math. Bull.} (2022). DOI: 10.4153/S0008439521001077.

\bibitem{Alkhaleefah-PAMS-2019} {\sc S. A. Alkhaleefah, I. R. Kayumov}, and {\sc S. Ponnusamy}, On the Bohr inequality with a fixed zero coefficient, {\it Proc. Amer. Math. Soc.} {\bf 147} (2019), 5263--5274. DOI: https://doi.org/10.1090/proc/14634.

\bibitem{Bayart-2014} {\sc F. Bayart, D. Pellegrino}, and {\sc J. B.-Sepulveda}, The Bohr radius of the $ n $-dimensional polydisk is equivalent to $\sqrt{(\log n)/n}$, {\it Adv. Math.}  {\bf 264} (2014), 726--746. https://doi.org/10.1016/j.aim.2014.07.029

\bibitem{Beneteau-2004} {\sc C. B${\rm \acute{E}}$n${\rm \acute{E}}$teau, A. Dahlner}, and {\sc D. Khavinson}, Remarks on the Bohr phenomenon, {\it  Comput. Methods Funct. Theory}  {\bf 4}  (2004), 1--19. https://doi.org/10.1007/BF03321051.

\bibitem{Ayt & Dja & BLMS & 2013} {\sc A. Aytuna} and {\sc P. Djakov}, Bohr property of bases in the space of entire functions and its generalizations, \textit{Bull. London Math. Soc.} \textbf{45} (2) (2013), 411--420.

\bibitem{Bers-BAMS-1997} L. Bers, Quasiconformal mappings, with applications to differential equations, function theory and topology, \textit{Bull. Amer. Math. Soc.} \textbf{83} (6)(1977), 1083-1100, doi:10.1090/S0002-9904-1977-14390-5.

\bibitem{Blasco-2010} {\sc  O. Blasco}, The Bohr radius of a Banach space. In: Curbera GP, Mockenhaupt G, Ricker WJ,editors. Vector measures, integration and related topics. Vol. 201, Operator theory and advanced applications. Basel: Birkhäuser; 2010. p. 59–64.  
https://doi.org/10.1007/978-3-0346-0211-2-5

\bibitem{Boas-1997} {\sc H. P. Boas} and {\sc D. Khavinson}, Bohr's power series theorem in several variables, {\it Proc. Amer. Math. Soc.}  {\bf 125} (1997), 2975--2979. https://doi.org/10.1090/S0002-9939-97-04270-6

\bibitem{Boas-2000} {\sc H. P. Boas}, Majorant Series,  {\it J. Korean Math. Soc.}  {\bf 37} (2000), 321--337.

\bibitem{Bhowmik-2018} {\sc B. Bhowmik} and {\sc  N. Das}, Bohr Phenomenon for subordinating families of certain univalent functions, {\it J. Math. Anal. Appl.} {\bf 462} (2018) 1087--1098.https://doi.org/10.1016/j.jmaa.2018.01.035

\bibitem{Bhowmik-Das-AM-2021} {\sc B. Bhowmik} and {\sc  N. Das}, A characterization of Banach spaces with nonzero Bohr radius. \textit{Arch. Math.} \textbf{116} (2021), 551-558.  https://doi.org/10.1007/s00013-020-01568-8.

\bibitem{Bohr-1914} {\sc H. Bohr}, A theorem concerning power series,  {\it Proc. Lond. Math. Soc}. s2-13 (1914), 1--5.

\bibitem{constantin-2017} {\sc A. Constantin} and {\sc M. J. Martin}, A harmonic maps approach to fluid flows, {\it Math. Ann.} {\bf 369} (2017), 1--16.

\bibitem{Choquet-1945} {\sc G. Choquet}, Sur un type de transformation analytique generalisant la representation conforme et definie
au moyen de fonctions harmoniques, \textit{Bull. Sci. Math.} \textbf{89}(1945), 156-165.


\bibitem{Das-JMAA-2022} {\sc N. Das}, Refinements of the Bohr and Rogosisnki phenomena, \textit{J. Math. Anal. Appl.} \textbf{508}(1) (2022), 125847.https://doi.org/10.1016/j.jmaa.2021.125847

\bibitem{Defant-2011} {\sc A. Defant} and {\sc L. Frerick}, The Bohr radius of the unit ball of $l^n_p$, {\it J. Reine Angew.  Math.}  {\bf 660} (2011), 131--147.https://doi.org/10.1007/s13348-016-0181-3

\bibitem{Dixon & BLMS & 1995} {\sc P. G. Dixon}, Banach algebras satisfying the non-unital von Neumann inequality, \textit{Bull. London Math. Soc.} \textbf{27}(4)(1995), 359--362. https://doi.org/10.1112/blms/27.4.359

\bibitem{Djakov-JA-20000} {\sc P. B. Djakov} and {\sc M. S. Ramanujan}, A remark on Bohr’s theorems and its generalizations, \textit{J. Analysis}, \textbf{8} (2000), 65-77.

\bibitem{Dorff-2003} {\sc M. Dorff}, Minimal graphs in $ \mathbb{R}^3 $ over convex domains, \textit{Proc. Amer. Math. Soc.} \textbf{132}(2003), 491-498.

\bibitem{Duren-1983} {\sc P. L. Duren}, Univalent 
Functions, {\it Springer, New York}, (1983).

\bibitem{Duren-Harmonic-2004} {\sc P. L. Duren}, Harmonic Mappings in the Plane. vol. 156, Cambridge Tracts in Mathematics, Cambridge
University Press, Cambridge, (2004).

\bibitem{Goluzin-Trudy-1951} {\sc G. M. Goluzin}, On subordinate univalent functions (Russian), {\it Trudy. Mat. Inst. Steklov} {\bf 38} (1951), 68-71.

\bibitem{Ponnusamy-RM-2021} {\sc S. Evdoridis}, {\sc S. Ponnusamy}, and {\sc A. Rasila}, Improved Bohr’s inequality for shifted disks, \textit{Results Math.} \textbf{76} (2021), 14.https://doi.org/10.1007/s00025-020-01325-x.

\bibitem{Grotzsch-1928} {\sc H. Gr\"otzsch},  \"Uber einige Extremalprobleme der konformen Abbildung. I, II., Berichte \"uber die Verhandlungen der K\"oniglich S\"achsischen Gesellschaft der Wissenschaften zu Leipzig. Mathematisch-Physische Classe (in German), \textbf{80}(1928): 367-376, 497-502.

\bibitem{Huang-Liu-Ponnu-AMP-2020} {\sc Y. Huang, M-S. Liu} and {\sc S. Ponnusamy}, Refined Bohr-type inequalities with area measure for	bounded analytic functions, \textit{Anal. Math. Phys.} (2020) 10:50.


\bibitem{Huang-Liu-Ponnu-CVEE-2021} {\sc Y. Huang, M.-S. Liu}, and {\sc S. Ponnusamy}, The Bohr-type operator on analytic functions and sections,  \textit{Complex Var. Elliptic Equ.} (2021) DOI: 10.1080/17476933.2021.1990272.

\bibitem{Ismagilov-2020-JMAA} {\sc A. Ismagilov, I. R. Kayumov}, and {\sc S. Ponnusamy}, Sharp Bohr type inequality, {\it J. Math. Anal. Appl.}  {\bf 489} (2020), 124147.https://doi.org/10.1016/j.jmaa.2020.124147

\bibitem{Ismagilov-2021-JMS} {\sc A. Ismagilov}, {\sc  A. V. Kayumova}, {\sc I. R. Kayumov}, and {\sc S. Ponnusamy}, Bohr inequalities in some classes of analytic functions, {\it J. Math. sci.}  \textbf{252} (3), (2021), DOI 10.1007/s10958-020-05165-6.

\bibitem{Kayu-Kham-Ponnu-2021-JMAA} {\sc I. R. Kayumov, D. M. Khammatova} and {\sc S. Ponnusamy},  Bohr-Rogosinski phenomenon for analytic functions and Ces\'aro operators, {\it J. Math. Anal. Appl.}  {\bf 496}(2) (2021), 124824.

\bibitem{Kayumov-CRACAD-2018} {\sc I. R. Kayumov} and {\sc S. Ponnusamy}, Improved version of Bohr’s inequality, C. R. Acad. Sci. Paris, Ser.I \textbf{356}(2018), 272--277. https://doi.org/10.1016/j.crma.2018.01.010.


\bibitem{Ponnusamy-2017} {\sc I. R. Kayumov} and {\sc S. Ponnusamy}, Bohr-Rogosinski radius for analytic functions, preprint, see https://arxiv.org/abs/1708.05585.

\bibitem{Kayumov-CMFT-2017} {\sc I. R. Kayumov} and {\sc S. Ponnusamy}, Bohr inequality for odd analytic  functions, {\it Comput. Methods Funct. Theory} \textbf{17} (2017), 679--688. https://doi.org/10.1007/s40315-017-0206-2.


\bibitem{Kayumov-2018-JMAA} {\sc I. R. Kayumov} and {\sc S. Ponnusamy}, Bohr's inequalities for the analytic functions with lacunary series and harmonic functions, {\it J. Math. Anal. Appl.}  {\bf 465} (2018), 857--871. https://doi.org/10.1016/j.jmaa.2018.05.038

\bibitem{Kay & Pon & AASFM & 2019} {\sc I. R. Kayumov} and {\sc S. Ponnusamy}, On a powered Bohr inequality, \textit{Ann. Acad. Sci. Fenn. Ser. A}, \textbf{44} (2019), 301--310. https://doi.org/10.5186/aasfm.2019.4416


\bibitem{kayumov-Ponnusamy-Shakirov-2017} {\sc I. R. Kayumov, S. Ponnusamy} and {\sc N. Shakirov}, Bohr radius for locally univalent harmonic mappings, {\it  Math. Nachrichten}  (2017).

\bibitem{Kumar-PAMS-2023} {\sc S. Kumar}, On the multidimensional Bohr radius, \textit{Proc. Amer. Math. Soc.} \textbf{151}(2023), 2001-2009.

\bibitem{Kumar-Sahoo-MJM-2021} {\sc S. Kumar} and {\sc S. Sahoo}, Bohr Inequalities for Certain Integral Operators, \textit{Mediterr. J. Math.} (2021) 18:268, https://doi.org/10.1007/s00009-021-01891-61660-5446/21/060001-12

\bibitem{Lata-Singh-PAMS-2022} {\sc S. Lata} and {\sc D. Singh}, Bohr's inequality for non-commutative Hardy spaces, \textit{Proc. Amer. Math. Soc.} \textbf{150}(1) (2022), 201-211. https://doi.org/10.1090/proc/15609

\bibitem{Lehto-Virtanen-1973} {\sc O. Lehto} and {\sc K. I. Virtanen}, Quasiconformal mapping in the plane, \textit{Springer-Verlag,} 1973.

\bibitem{Lew-BAMS-1936} {\sc H. Lewy}, On the non-vanishing of the Jacobian in certain in one-to-one mappings, Bull. Amer. Math. Soc. \textbf{42} (1936), 689--692.

\bibitem{Liu-JMAA-2021} {\sc G. Liu}, Bohr-type inequality via proper combination, \textit{J. Math. Anal. Appl.} \textbf{503}(1) (2021) : 125308.https://doi.org/10.1016/j.jmaa.2021.125308

\bibitem{Liu-Liu-Ponnusamy-2021} {\sc G. Liu}, {\sc Z. Liu}, and {\sc S. Ponnusamy}, Refined Bohr inequality for bounded analytic functions, \textit{Bull. Sci. Math.} \textbf{173} (2021), 103054.https://doi.org/10.1016/j.bulsci.2021.103054

\bibitem{Liu-Ponnusamy-BMMS-2019} {\sc Z. Liu} and {\sc S. Ponnusamy}, Bohr radius for subordination and k-quasiconformal harmonic mappings, \textit{Bull. Malays. Math. Sci. Soc.} \textbf{42} (2019) 2151–2168.

\bibitem{Liu-Ponnusamy-PAMS-2021} {\sc M-S. Liu} and {\sc S. Ponnusamy}, Multidimensional analogues of refined Bohr's inequality, \textit{Porc. Amer. Math. Soc.} \textbf{149}(5), (2021), 2133-2146. https://doi.org/10.1090/proc/15371

\bibitem{Liu-Shang-Xu-JIA-2018} {\sc M-S. Liu, Y-M. Shang}, and {\sc J-F. Xu}, Bohr-type inequalities of analytic functions,\textit{ J. Inequal. Appl.} (2018) 2018:345.. https://doi.org/10.1186/s13660-018-1937-y

\bibitem{Paulsen-PLMS-2002} {\sc V. I. Paulsen, G. Popescu}, and {\sc D. Singh}, On Bohr’s inequality, \textit{Proc. Lond. Math. Soc.} \textbf{85}(2) (2002) 493–512. DOI: https://doi.org/10.1112/S0024611502013692

\bibitem{Paulsen-PAMS-2004} {\sc V. I. Paulsen} and {\sc D. Singh}, Bohr’s inequality for uniform algebras. \textit{Proc. Amer. Math. Soc.} \textbf{132} (2004), 3577–3579. https://www.jstor.org/stable/4097340

\bibitem{Paulsen-BLMS-2006} {\sc V. I. Paulsen} and {\sc D. Singh}, Extensions of Bohr’s inequality, \textit{Bull. Lond. Math. Soc.} \textbf{38}(6) (2006) 991–999. DOI: https://doi.org/10.1112/S0024609306019084.

\bibitem{Pon-Shm-Star-JMAA-2024} {\sc S. Ponnusamy, E. S. Shmidt} and {\sc V. V. Starkov}, The Bohr radius and its modifications for linearly invariant families of analytic functions, \textit{J. Math. Anal. Appl.} \textbf{533}(1)(2024), 128039.

\bibitem{Ponnusamy-RM-2020} {\sc S. Ponnusamy, R. Vijayakumar}, and {\sc K.-J. Wirths}, New inequalities for the coefficients of unimodular bounded functions, \textit{Results Math.} \textbf{75} (2020): 107. https://doi.org/10.1007/s00025-020-01240-1

\bibitem{Ponnusamy-HJM-2021}{\sc S. Ponnusamy, R. Vijayakumar}, and {\sc K.-J. Wirths}, Modifications of Bohr’s inequality in various	settings, \textit{Houston J. Math.} (2021), To appear. See also  https://arxiv.org/pdf/2104.05920.pdf.

\bibitem{Ponnusamy-JMAA-2022} {\sc S. Ponnusamy, R. Vijayakumar}, and {\sc K.-J. Wirths}, Improved Bohr’s phenomenon in quasi-subordination classes, \textit{J. Math. Anal. Appl.} \textbf{506} (1) (2022), 125645, 10 pages. https://doi.org/10.1016/j.jmaa.2021.125645

\bibitem{Rogosinski-1923} {\sc W. Rogosinski}, \"Uber Bildschranken bei Potenzreihen und ihren Abschnitten, \textit{Math. Z.}, \textbf{17} (1923), 260–276. https://doi.org/10.1007/BF01504347.


\bibitem{Sidon-1927} {\sc S. Sidon}, Uber einen satz von Hernn Bohr, {\it Math. Zeit.}  {\bf 26} (1927),  731-732.https://doi.org/10.1007/BF01475487

\bibitem{Tomic-1962} {\sc M. Tomic}, Sur un Theoreme de H. Bohr, {\it Math. Scand.} {\bf 11} (1962) 103--106. https://www.jstor.org/stable/24489314


\bibitem{Schur-1925} {\sc I. Schur} and {\sc G. Szeg\"o}, \"Udie Abschnitte einer im Einheitskreise beschrankten Potenzreihe, \textit{Sitz. ber. Preuss. Akad. Wiss. Berl. Phys.-Math. Kl.} (1925) 545-560.
	
\bibitem{Vuorinen-1988} {\sc M. Vuorinen}, Conformal geometry and quasiregular mappings, Lecture Notes Math., vol.1319, {\it Springer-Verlag}, 1988.   
\end{thebibliography}
\end{document}